\documentclass[11pt, oneside]{article}   	
\usepackage{geometry}                		
\usepackage{graphicx}				
\usepackage{amssymb}
\usepackage{amsthm}
\usepackage{comment}
\usepackage{paralist}

\usepackage{capt-of}

\newcommand{\etal}{et al.}
\newcommand{\apriori}{a priori}

\newcommand{\kmeans}{k-means}
\usepackage{hyperref}
\usepackage{jpfairbanks}

\newcommand{\roc}{ring of cliques}
\newcommand{\Roc}{{\cal R}_{b,q}} 

\newcommand{\cut}[1]{E\paren{S,\bar{S}}}
\newcommand{\conductance}[1]{\operatorname{\phi}\left({#1}\right)}
\newcommand{\condG}{\phi_G}
\newcommand{\sweepcut}[2]{S_{#1}\paren{#2}}

\newcommand{\vol}[1]{\operatorname{Vol}\left(#1\right)}

\newcommand{\pairmin}[2]{\min \paren{{#1},{#2}}}
\newcommand{\pairmax}[2]{\max \paren{{#1},{#2}}}

\newcommand{\rayquot}{Rayleigh quotient}

\newcommand{\eres}{eigenresidual}

\newcommand{\evecs}{eigenvectors}

\newcommand{\ray}[2]{{#1}^T #2 #1}

\newcommand{\rtwo}{\sqrt{2}}

\newcommand{\ones}{\mathbf{1}}

\newcommand{\clintervalcl}[2]{\bracket{#1,#2}}
\newcommand{\opintervalop}[2]{\paren{#1,#2}}
\newcommand{\spectralgap}{\delta_1}

\newcommand{\blendinterval}{\clintervalcl{\lambda_q}{\lambda_p}}

\newcommand{\ahat}{{\hat{A}}}
\newcommand{\ahatshifted}{\ahat - \bk \bk^t}

\DeclareMathOperator{\spann}{Span}

\title{Spectral Partitioning with Blends of Eigenvectors}
\author{James P Fairbanks \and Geoffrey D Sanders \and David A Bader}
\date{\today{}}
\usepackage{cleveref}
\newtheorem{thm}{Theorem}
\newtheorem{lemma}{Lemma}

\newtheorem{coro}{Corollary}
\newtheorem{rem}{Remark}

\newcommand{\beq}{\begin{equation}}
\newcommand{\eeq}{\end{equation}}
\newcommand{\bdm}{\begin{displaymath}}
\newcommand{\edm}{\end{displaymath}}
\newcommand{\beqa}{\begin{eqnarray}}
\newcommand{\eeqa}{\end{eqnarray}}
\newcommand{\beqas}{\begin{eqnarray*}}
\newcommand{\eeqas}{\end{eqnarray*}}
\newcommand{\barr}{\begin{array}}
\newcommand{\earr}{\end{array}}
\newcommand{\bit}{\begin{itemize}}
\newcommand{\eit}{\end{itemize}}
\newcommand{\qq}[1]{\qquad \mbox{#1} \qquad}

\newcommand{\cG}{{\cal G}}

\newcommand{\cN}{{\cal N}}
\newcommand{\cO}{{\cal O}}

\newcommand{\cV}{{\cal V}}
\newcommand{\cX}{{\cal X}}

\newcommand{\bc}{{\bf c}}

\newcommand{\be}{{\bf e}}

\newcommand{\bg}{{\bf g}}
\newcommand{\bh}{{\bf h}}

\newcommand{\bk}{{\bf k}}

\newcommand{\bp}{{\bf p}}
\newcommand{\bq}{{\bf q}}

\newcommand{\bs}{{\bf s}}

\newcommand{\bv}{{\bf v}}

\newcommand{\bx}{{\bf x}}
\newcommand{\by}{{\bf y}}
\newcommand{\bz}{{\bf z}}
\newcommand{\bfo}{{\bf 1}}

\newcommand{\Lh}{\hat{L}}
\newcommand{\Ah}{\hat{A}}

\newcommand{\amat}{M}
\newcommand{\Dneghalf}{D^{-\half}}
\newcommand{\Dhalf}{D^{\half}}

\newcommand{\fiedval}{\lambda_2\paren{L}}
\newcommand{\ritzval}{\mu}
\newcommand{\fiedvec}{\bv_2\paren{L}}

\newcommand{\Pipq}{\Pi}

\usepackage{xfrac}
\renewcommand{\cref}{\Cref}

\begin{document}

\maketitle
\begin{abstract}
Many common methods for data analysis rely on linear algebra.
We provide new results connecting data analysis error to numerical accuracy in the context of spectral graph partitioning.
We provide pointwise convergence guarantees so that
spectral blends (linear combinations of eigenvectors) can be employed to solve data analysis problems with confidence in their accuracy.
We apply this theory to an accessible model problem, the ring of cliques, by deriving the relevant eigenpairs and
  finding necessary and sufficient solver tolerances.
Analysis of the ring of cliques provides an upper bound on eigensolver tolerances for graph partitioning problems.
These results bridge the gap between linear algebra based data analysis methods and the convergence theory of iterative approximation methods.
These results explain how the combinatorial structure of a problem can be recovered much faster than numerically
accurate solutions to the associated linear algebra problem.
\end{abstract}


\section{Introduction}

Spectral methods are a valuable tool for finding cluster structure in data.
While all spectral methods rely on approximating the eigenvectors of a matrix, the impact of numerical accuracy on 
the quality of the partitions is not fully understood.
Spectral partitioning methods proceed in two steps, first one or more vectors approximating eigenvectors of a graph matrix are computed, and then a partitioning scheme is applied to those vectors.
While many theoretical results quantify the relationship between the exact solution to the numerical problem and the solution to the original data mining problem,
  few address data analysis errors introduced by error in the numerical solution.
For instance \cite{kannan2004goodbadspectral} studies the runtime and quality (in terms of conductance) of partitioning algorithms including spectral methods.
Often the eigenvector computation is used as a primitive operation without accounting for the trade-off between run time and numerical accuracy.
Guattery and Miller~\cite{guattery1998quality} studies various methods of applying exact eigenvectors to partition graphs by producing examples where each variation does not find the optimal cut.
Our paper addresses the effect of numerical error in the eigenvector computation on the quality of sweep cuts which reveal graph structure.

In order to understand the impact of numerical error on spectral partitioning, we study both general matrices and a
specific family of graphs.
Finding error and residual tolerances for general graphs is a difficult problem.
\Cref{sec:blends} provides tools for deriving a residual tolerance for arbitrary graphs.
\Cref{sec:ROCstart} analyzes a model problem with clear cluster structure, where linear combinations of eigenvectors
represent a space of multiple good partitions,
and applies \Cref{sec:blends} results to derive a residual tolerance sufficient for solving this model problem.
This use of a model problem is well established in the linear algebra literature where the Laplace equation on a regular grid
is common in papers and software regarding the solution of systems of equations.
Analysis of this model problem allows us to derive a solver tolerance for correctly recovering the clusters with a sweep cut scheme.
This analysis illustrates the difference between accurately solving the equation and correctly recovering the combinatorial structure.

This approach to approximate eigenvectors can be applied to other applications where a numerical method solves a data mining problem,
  such as solving personalized Pagerank as a linear system~\cite{delcorso2005pageranklinearsystem} to rank vertices in a graph,
  or evaluating commute times~\cite{DoyleSnell1984} to produce a metric distance on the vertices.
These methods also apply numerical solvers to infer a combinatorial or data analysis structure from the graph. 
A similar treatment, in terms of a model problem, of these methods would benefit our understanding of the relationship between
numerical accuracy and data analysis accuracy.

 Here we introduce the necessary concepts of data analysis quality and eigensolver accuracy.
For this work we focus on partitioning graphs to minimize conductance as defined below.
For any $S\subset V$, $S\cup \bar{S}=V$ represents a cut of the graph.
Define $\vol{S} = \sum_{i,j\in S} a_{i,j}$ as the total weight of the edges within $S$.
And define $\cut{S} = \sum_{i\in S, j\notin S} a_{i,j}$ as the total weight of edges with one
vertex in $S$ and one vertex in the complement of $S$.
The conductance of a cut $S$ is thus given by the formula~\cite{kannan2004goodbadspectral}:
\[
    \conductance{S} = \frac{\cut{S}}{\pairmin{\vol{S}}{\vol{\bar{S}}}}.
\]
For any vector $\bx$, represent the sweep cut of $\bx$ at $t$ as in \cref{eqn:sweepcut}.
\begin{equation}\label{eqn:sweepcut}
    S=\sweepcut{\bx}{t} = \set{i\mid x_i > t}
\end{equation}
We denote by $\conductance{\bx}$ the minimal conductance of a sweep cut of $\bx$, that is $\min_t \conductance{S_x\paren{t}}$.
The conductance of the graph is defined as $\condG = \min_S \conductance{S}$.
If the graph has multiple partitions with conductance less than a value $\psi$, then the application might accept any of them.

The accuracy of a solution to the eigenvector problem can be measured in three quantities: \rayquot{}, error, and residual.
Spectral methods for finding low-conductance partitions rely on computing vectors $\bx$ and corresponding scalars $\lambda$ that solve the equations $\amat \bx = \lambda \bx$ for some graph-associated matrix $\amat$.
The \rayquot{}, $\ritzval = \bx^t \amat \bx$ is an approximation to the eigenvalue $\lambda$.
The error $\norm{\bv-\bx}$ where $\bv$ is the closest exact solution 
is not accessible to a solver in general.
The solver can use the norm of the eigenresidual, $\norm{\amat\bx - \ritzval\bx}$, to determine when to stop iterations.
Throughout this paper $\norm{\cdot}$ will
be taken to mean the 2-norm with subscripts used to clarify when necessary.
In order to practically use an eigensolver, one must choose a residual tolerance $\norm{\amat\bx-\ritzval\bx} <
\epsilon$ sufficient to ensure that the computed eigenvector is accurate enough to solve the application problem.
This paper provides concrete residual tolerances for a specific model problem and provides tools for finding such
tolerances for more general graphs.

We briefly summarize notation for various graph matrices. 
Let $\ones$ be the all ones vector.
If 
  $A$ is the adjacency matrix of a graph and 
  $D$ is the diagonal matrix whose entries are $d_{i,i}=\paren{A\ones}_i$,
  then $L=D-A$ is the combinatorial Laplacian,
  and $\Lh = I - \Dneghalf A \Dneghalf$ is the normalized Laplacian. 
Solutions of the generalized eigenvector problem $L\by=D\by$ are scaled solutions to the eigenvector problem $\Dneghalf A \Dneghalf \bx = \bx$ where the scaling is $\bx = \Dhalf \by$. 
We refer to $\Dneghalf A \Dneghalf$ as $\Ah$, and use the identity
$\lambda_k(\Lh) = 1-\lambda_{n-k}(\Ah)$ to replace computations involving small eigenvalues of the normalized Laplacian
matrix with computations involving large eigenvalues of the adjacency matrix.

Conductance is an appropriate measure of partition quality for spectral partitioning because of Cheeger's inequality
which bounds the conductance of the graph in terms of the eigenvalues of the Laplacian matrix.
\renewcommand{\amat}{\Lh}
\begin{thm}{General Cheeger Inequality\cite{mihail1989conductance}}\label{thm:cheeger}
    If $\bx$ is a unit vector orthogonal to $D^{\half}\ones$ such that $\ray{\bx}{\amat} = \ritzval$ then $\Dneghalf\bx$ has a sweep cut $S$ such that
  $\conductance{S} = \conductance{\bx} \leq \sqrt{2\ritzval}$.
\end{thm}
When $\bx$ satisfies $\Lh\bx=\lambda_2\bx$,
    $\condG \le \conductance{\Dneghalf\bx} \le \sqrt{2\fiedval}$.
This general form of Cheeger's inequality indicates that finding low-energy Laplacian eigenvectors is sufficient for
constructing low-conductance partitions of the graph.

In graph partitioning, the goal is to compute a partition of the graph that optimizes the chosen objective.
When applying spectral methods to graph partitioning, our goal is not to compute very accurate eigenpairs, but instead to partition the vertex set of a graph correctly.
Because spectral partitioning can be used recursively to find small clusters,
  we focus on splitting a graph into two parts.
Our results on the model problem indicate that approximate eigenvectors are sufficient to solve the data analysis
problem and are much faster to compute if the graph has the right structure.

\subsection{A Model Problem}\label{sec:rocjustification}

We use a simple model (the \roc{}) to study the capabilities of spectral partitioning algorithms, form theory to
characterize performance, and potentially enhance these algorithms.   Such use of model problems is well-established in
the numerical analysis literature regarding iterative solutions to discretized partial differential equations (PDEs).
The Dirichlet Laplacian  on a unit square discretized on a Cartesian lattice is a simple problem with known eigenpairs
and is used to study the properties of various eigensolvers.    These simple model problems do not demonstrate the
algorithms perform well on real-world problems, but are incredibly important tools for algorithm development and theoretical analysis.   For spectral partitioning, the \roc{} is one candidate model problem for which we can derive complete knowledge of the eigenpairs.
In PDEs, the order of discretization error (difference between continuous solution and discrete solution) provides the solver with a stopping criterion.   In spectral graph partitioning, we do not
have this luxury, and we must develop theory to understand how perturbations in a spectral embedding impact partitioning quality.
Another reason to develop a collection of model problems
is to enable careful study of this impact in well-understood situations.

In order to provide a striking example of our improved analysis \Cref{sec:ROCstart} studies our model problem in detail.
The goal is to understand when approximate eigenvectors have sweep cuts that correctly identify graph structures.
The \roc{} has been studied  as 
  ``the most modular network'' in order to demonstrate a resolution limit in the 
  modularity maximization procedure for community detection~\cite{Fortunato2007resolutionlimit}. 
For this family of highly structured graphs, the correct partition is unambiguous.
We use the \roc{} to investigate how spectral embeddings for partitioning are affected by numerical error.
Because of the high degree of symmetry, the \roc{} allows for a thorough closed form analysis producing formulas for the eigenvectors and eigenvalues.
A sweep cut using exact eigenvectors partitions the graph with small conductance and successful recovery of all the
clusters.
We quantify the effects of approximation error on sweep cut partitions of this graph.
Our findings demonstrate that despite a small spectral gap, which implies slow convergence of non-preconditioned eigensolvers, the \roc{} is well
partitioned by low accuracy approximations to the eigenvectors.


Studying the \roc{} provides guidance for practitioners on useful tolerances for eigensolvers.
We are able to construct the smallest perturbation that induces a mistake in the sweep cut partition.
This perturbation shows that when looking for clusters of size $b$ in a general graph the eigensolver tolerance must be
smaller than $\bigo{b^{-\half}}$.
Analysis of the \roc{} provides an upper bound on the eigensolver accuracy that is 
sufficient to recover community structure.

\subsection{Contributions}\label{sec:contributions}
This paper provides the following contributions.
\Cref{sec:blends} extends a known error bound on the computation of eigenvectors to the computation of linear combinations of eigenvectors.
By extending this classical error bound to linear combinations of eigenvectors,
  we find a condition on the spectrum of where numerically accurate blends are easy to achieve.
\Cref{thm:blendspaceconductance} provides a general condition under which approximate eigenvectors preserve sweep cuts.
\Cref{sec:ROCstart} analyzes a model problem and derives necessary and sufficient error tolerances for solving the model problem, which are essentially tight for some parameter regime.
We show for the model problem where the number of clusters is polynomial in the size of the clusters, the power method takes $\bigo{1}$ iterations to identify the clusters.

\subsection{Related Work}\label{sec:relatedwork}


Iterative methods for solving eigenvector problems are well understood.
These algorithms are able to generate solutions to arbitrary approximation factors,
  but have run time which increases in the number of iterations,
  where more iterations leads to better approximations.
Iterative methods~\cite{arpack,saad2011numerical} have been shown to provide fast approximate solutions for a wide range
of problems.
Many iterative eigensolvers can be represented as the output $\by$ equals a polynomial $p$ applied to the matrix $M$ times a
vector $\bx$, $\by = p(M)\bx$.
The degree of $p$ depends on the number of iterations of the method, which is controlled by the eigenresidual tolerance
$\norm{M\by-\mu\by} < \epsilon$.
The simplest such method is the power method which is easy to analyze because $p(M)$ is always $M^k$ where $k$ is the number of iterations.   More sophisticated methods choose $p(M)$ adaptively and typically converge more quickly.
A practical implementation of the Arnoldi method can be found in~\cite{arpack}, which is commonly used in practice.

Localized eigenvectors are essential to analysis of the \roc{}.
Cucuringu and Mahoney examine the network analysis implications of localized interior eigenvectors in the spectrum of the Co-voting network of US Senators~\cite{cucuringu2011localization}.
The graph is defined 
    with connections between members of the same session of Congress who vote together on the same bills 
    and connections between individuals who are reelected to consecutive sessions.
The first 41 eigenvectors are oscillatory across the congressional sessions with little variation between the vertices in the same session,
    but the next eigenvectors are small in magnitude on most sessions but take large positive values on members of one party and large negative values on members of the other party within a few sessions.
Thus blends of the dominant eigenvectors indicate the sessions of congress.
The \roc{} also exhibits globally periodic extremal eigenvectors and localized interior eigenvectors due to its
Kronecker product structure. 
We show that the \roc{}, has a basis for an interior eigenspace with the nonzero entries of each vector completely restricted to an individual clique.
This localization allows us to show that approximate eigenvectors recover the interesting combinatorial structure.

Other work focuses on the impact of errors in measurement on the behavior of data analysis algorithms.
In the context of Gram (kernel) matrices,
Huang \etal\cite{huang2009spectral-clustering-with-perturbed-data}, studies the effect of perturbing the original data points on the spectral partitioning method.
A similar line of investigation is pursued in \cite{yan2009fastapproximate} where data points are quantized to reduce bandwidth in a distributed system.
This work connects approximation with performance.
If one can demonstrate that data analysis accuracy is not affected too much,
  then one can use an algorithm which sacrifices accuracy to improve performance.
Our paper treats the data as correctly observed and handles error in the iterative solver.

The impact of approximate numerical computing has been shown useful for several applications.
In~\cite{boutsidis2015spectralclustering} eigenvectors of a kernel matrix are approximated with the power method and then \kmeans{} is applied to these approximations.
The \kmeans{} objective function is well approximated when using approximate eigenvectors.
The bounds given in \cite{boutsidis2015spectralclustering} depend on using the $k$ \evecs{} to partition into $k$ parts
and depend on the $k$th spectral gap to control accuracy of approximation.
Experiments also show that \kmeans{} on the approximate eigenvectors is faster and sometimes more accurate in terms of Normalized Mutual Information (NMI) compared to using exact eigenvectors.
Our paper focuses on partitioning into two clusters based on sweep cuts of a single approximate eigenvector and makes a
rigorous analysis of a model problem in order to understand how the numerical accuracy interacts with combinatorial
structure of the data clusters.
Pothen \etal{} \cite{pothen1990partitioning}, which used spectral partitioning for distributed memory sparse matrix computation, recognized the value of low-accuracy solutions.
Approximate spectral coordinates are used to reorder matrices before conducting high accuracy linear solves.
Our paper contributes to the understanding of how numerical approximation accuracy contributes to data analysis accuracy.

\renewcommand{\amat}{M}

\section{Blends of Eigenvectors}\label{sec:blends}

In order to understand the relationship between eigensolver error and graph partitioning, we study error bounds and the effect of pointwise error on the sweep cut procedure.
Theorems~\ref{thm:singleevecerror} and~\ref{thm:blenderror} bound the error to a subspace in terms of the residual and quantities derived from the eigenvalues of the matrix.
This control over the error is then used in \Cref{thm:blendspaceconductance} to relate eigenresidual to the conductance of a sweep cut of the graph.
These results apply to general matrices. 
Although a small spectral gap implies poor control on the error to a single eigenvector,
we derive a condition where low accuracy approximations effectively partition the graph.
Section~\ref{sec:ROCstart} applies these theorems to a special family of graphs to show that blends are faster to
compute and provide nearly optimal partitions.


\subsection{Converging to a Single Eigenspace}\label{sec:convergesingle}
Let $\amat \in \mathbb{R}^{n \times n}$, $\amat = \amat^t$ be a general symmetric matrix.
Consider the solutions to the equation $\amat \bv = \lambda \bv$.
Because $\amat$ is symmetric, there are $n$ eigenvalues in $\mathbb{R}$ (counting multiplicities).
The set of all eigenvalues is the spectrum $\lambda(\amat)$, which we order decreasingly as $\lambda_1 \geq \lambda_2 \geq \cdots \geq \lambda_n$.
For $k=1,\dots,n$,  let $\bv_k$ be an eigenvector associated with $\lambda_k$, $\amat \bv_k = \lambda_k \bv_k$, such that $\bv_k^t \bv_l = 0$ whenever $l \neq k$.
Define the eigenspace associated with $\lambda_k$ as the invariant subspace associated with $\lambda_k$, that is
\[
\cX_k := \left\{ \, \bx \in \mathbb{R}^n \,  : \,  \amat \bx = \lambda_k \bx \, \right\}.
\]
These definitions imply $\mbox{dim}(\cX_k) = \mbox{mult}(\lambda_k)$ and $\cX_k = \cX_l$ when $\lambda_k = \lambda_l$.

\begin{rem}
The results in this section are stated and proved in terms of generic symmetric matrix $\amat$ because they apply beyond spectral graph theory.
Spectral partitioning methods use the \evecs{} of $\Lh = I-\ahat$.
Counting eigenvalues from largest to smallest starting with 1, we see $\lambda_2\paren{\ahat} = \lambda_{n-1}\paren{\L}$ with the same eigenvectors.
Letting $\bk = D^{1/2} \bfo\norm{D^{1/2} \bfo}^{-1}$, the normalized eigenvector of $\ahat$ associated with $\lambda_1\paren{\ahat}=1$, one can use $\amat = \ahatshifted$ in the results of this section.
Subtracting $\bk \bk^t$ moves the eigenvalue $1$ to $0$, or $\bk$ is an eigenvector associated with $0 \in \lambda\paren{\ahatshifted}$.
Thus, for this $\amat$ and for $k$ where $\lambda_k\paren{\amat}>0$, we have
$$
  \lambda_k\paren{\amat} = \lambda_{k+1}\paren{\ahat}.
$$
In particular, for the Fiedler eigenvalue, $\lambda_1\paren{\amat} = \lambda_{2}\paren{\ahat}=1-\lambda_{n-1}\paren{\Lh}$,
 and the Fiedler vectors in the associated eigenspace correspond to extremal eigenvalue $\lambda_1\paren{\amat}$.
Computationally, implementations of iterative methods approximating the eigenvectors of $\lambda_2\paren{\ahat}$ perform
better with a routine applying the operator $\ahatshifted$.
\end{rem}

Let $(\bx, \mu)$, $\bx \in \mathbb{R}^n$, $\mu \in \mathbb{R}$, be an approximate eigenpair of $\amat$ with $\norm{\bx}
= 1$ and $\ritzval = \ray{\bx}{\amat}$, the {\em \rayquot{}}, which minimizes the function $\norm{\amat\bx - \theta \bx}$ over all real values $\theta$.   Define the two-norm of the eigenresidual as $\epsilon = \norm{\amat \bx - \mu \bx}$.
As in~\cite{parlett1998symevp}, we have a simple eigenvalue bound.
By decomposing $\bx$ in the eigenbasis $\bx = \sum_{k=1}^n \alpha_k \bv_k$, we see
$$
  \epsilon^2 = \| \amat \bx - \mu \bx \|^2 = \sum_{k=1}^n \alpha_k^2 (\lambda_k - \mu)^2 \geq
  \left( \sum_{k=1}^n \alpha_k^2 \right) \left(\min_{1 \leq k \leq n} (\lambda - \mu)^2 \right)
$$

meaning there exists an eigenvalue $\lambda_k$ within $\epsilon$ of $\mu$,

$$
  \min_{1\leq k \leq n} |\lambda_k - \mu| \leq \epsilon.
$$

Also in~\cite{parlett1998symevp}, we have bounds estimating convergence to an eigenspace in angle.
Define the eigengap for $\lambda_k$ as $\delta_k = \min_{\lambda \in \lambda\paren{\amat} \setminus \lambda_k} \abs{\lambda_k - \lambda}$.   Moreover, if $\epsilon$ is small compared to $\delta_k$ there exists a normalized eigenvector $\bv \in \cX_k$ with which $\bx$ has a small angle,

$$
  \min_{\bv \in \cX_k} \sqrt{1 - \left< \, \bx, \, \bv \,\right> ^2} \leq \frac{\epsilon}{\delta_k}.
$$

Instead of presenting a proof of this well-known result, we derive a similar bound for $\ell_2$ and point-wise approximation to an eigenspace associated with an extremal eigenvector.

\newcommand{\closestmucondition}{|\mu - \lambda_1| < \min_{\lambda \in \lambda(\amat) \setminus \lambda_1} |\mu - \lambda|}
\begin{thm}\label{thm:singleevecerror}
Consider approximate eigenpair $(\bx, \mu)$ of symmetric $\amat \in \mathbb{R}^{n \times n}$ with $\norm{\bx} = 1$ and $\mu = \ray{\bx}{\amat}$.
Assume $$\closestmucondition.$$
Given eigenresidual $\epsilon$ and eigengap $\delta_1$, there exists an eigenvector $\bv \in \cX_1$ , with $\norm{\bv}=1$, and error bound
$$\norm{\bx - \bv} \leq \frac{\sqrt{8}\epsilon}{\delta_1}.  $$
\end{thm}

\begin{proof}
Let $\alpha \bv$ be the closest vector in $\cX_1$ to $\bx$.
Decompose $\bx$ into its eigenvector components within $\cX_1$ and perpendicular to $\cX_1$, $\bx = \alpha \bv + \sum_{\bv_k \perp \cX_1} \alpha_k \bv_k$.
Because $\mu$ is closer to $\lambda_1$ than any other eigenvalue, we have
\beqas
\epsilon^2 & = & |\lambda_1- \mu|^2 \alpha^2 + \sum_{\bv_k \perp \cX_k}  |\lambda_k - \mu|^2 \alpha^2_k\
\geq \sum_{\bv_k \perp \cX_k}  |\lambda_k - \mu|^2 \alpha^2_k
\geq \frac{\delta_1^2}{4} \sum_{\bv_k \perp \cX_k}  \alpha^2_k. \\
\eeqas
Rearranging gives
\beqas
\frac{4 \epsilon^2}{\delta_1^2} & \geq & \sum_{\bv_k \perp \cX_k}  \alpha^2_k = 1- \alpha^2. \\
\eeqas

\beqas
\|\bx-\bv\|^2 & = & \|\bx - \alpha\bv \|^2 + \|\bv - \alpha \bv\|^2
\leq  \|\bx - \alpha\bv \|^2 + (1 - \alpha)^2
\leq  2(1-\alpha^2)
\leq  \frac{8 \epsilon^2}{\delta_1^2} ,
\eeqas
\end{proof}

This result implies,
$$\frac{1}{n} \sum_{i\in\cV} |x_i - v_i|^2 \leq \|\bx - \bv\|^2 \leq \frac{8 \epsilon^2}{n \delta_1^2}  $$
so if $\epsilon^2 / n$ is small compared to the $\delta_1^2$, then the average average error squared is also small.
Moreover we have a point-wise error bound,
$$\max_{i\in\cV} |x_i - v_i| \leq \|\bx - \bv\| \leq \frac{\sqrt{8}\epsilon}{\delta_1}   $$

For a large graph, it is typical that Fiedler eigenvalue is so close to the next-to-largest eigenvalue, that the error bounds demand an extremely small eigenresidual for convergence.   
Note that this error analysis is independent of the algorithm used to compute the solution.
Choosing $\bx=\bv_3$ shows that the condition $\closestmucondition$ is necessary. 
Thus for matrices $\amat$ with small $\spectralgap$, we need to remove this condition.
When the $\spectralgap$ is small, a reasonable number of iterations of an eigensolver may
produce a vector $\bx$ with $\ray{\bx}{\amat}$ close to $\lambda_1$ which may not be very close to the true extremal eigenspace.
In this case we examine convergence to a linear combination of eigenvectors associated with a range of eigenvalues.

\subsection{Converging to a Subspace Spanned by Multiple Eigenspaces}\label{sec:convergeblends}
\newcommand{\blendgap}{\delta_{p,q}(\mu)}
\newcommand{\blendgapdef}{\min_{\{ k < p\} \cup \{ k > q \}} |\lambda_k - \mu|}


This section generalizes the previous error bound to the distance between $\bx$ and a subspace spanned by the eigenvectors associated with a range of eigenvalues.
Assume that linear combinations of eigenvectors associated with a range of eigenvalues $\blendinterval$ are satisfactory for subsequent data analysis algorithms.
If the \rayquot{} is within $\blendinterval$ and the eigenresidual is smaller than the distance between $\mu$ and any eigenvalue outside $\blendinterval$, then the following theorem holds.

\begin{thm}\label{thm:blenderror}
Consider approximate eigenpair $(\bx, \mu)$ of symmetric $\amat \in \mathbb{R}^{n \times n}$  with $\|\bx\| = 1$ and $\mu = \bx^t \amat \bx \in \blendinterval$.
Define
\[
\blendgap = \blendgapdef
\qquad \mbox{and} \qquad
\cX_p^q := \bigotimes_{k=p}^q \cX_k.
\]
Given eigenresidual $\epsilon = \|\amat \bx - \mu \bx \| \le \pairmin{\mu-\lambda_{q-1}}{\lambda_{p+1}-\mu}$
there exists a vector $\bv \in \cX_p^q$, with $\|\bv\|=1$,
$\ell_2$ error bound,
\[
\|\bx - \bv\| \leq \frac{\sqrt{2} \epsilon}{\delta_{p,q}(\mu)},
\]
\end{thm}

\begin{proof}
Let
$\bx = \sum_{k=1}^n \alpha_k \bv_k$ and $\Pipq \bx = \sum_{k=p}^q \alpha_k \bv_k$, the $\ell_2$-orthogonal projection onto $\cX_p^q$.   Note $\sum_{k=1}^n \alpha_k^2 = 1$.    In the case where $\|\Pipq \bx\| = 1$, we can let $\bv = \bx$ and see the bound is clearly satisfied.   For $\| \Pipq \bx \| < 1$, we first demonstrate that $\|\bx - \Pipq \bx\|$ is controlled by $\epsilon,$

\beqas
\epsilon^2 & = & \|\amat \bx - \mu \bx \|^2  \\
\epsilon^2 & = & \sum_{k=1}^n \alpha_k^2 |\lambda_k - \mu|^2 \\
\epsilon^2 & = & \sum_{\{k< p\} \cup \{k > q\}}  \alpha_k^2 |\lambda_k - \mu|^2 + \sum_{k=p}^q \alpha_k^2 |\lambda_k - \mu|^2 \\
\epsilon^2 & \geq &  \delta_{p,q}(\mu)^2 \sum_{\{k< p\} \cup \{k > q\}}  \alpha_k^2  + \sum_{k=p}^q \alpha_k^2 |\lambda_k - \mu|^2 \\
\epsilon^2 & \geq &  \delta_{p,q}(\mu)^2 \sum_{\{k< p\} \cup \{k > q\}}  \alpha_k^2  \\
\epsilon^2 & \geq &  \delta_{p,q}(\mu)^2 \left( 1 - \sum_{k=p}^q \alpha_k^2 \right)  \\
\frac{\epsilon}{\delta_{p,q}(\mu)} & \geq & \|\bx - \Pipq \bx \|.
\eeqas

There is a unit vector in $\cX_p^q$ that is also within some factor of $\epsilon$ to $\bx$.   Let $\bv = \Pipq \bx / \|\Pipq \bx\|$, then $ \|\bv - \Pipq \bx \| = 1 - \|\Pipq \bx \| $.   We have

\beqas
\| \Pipq \bx \|^2 + \|(I - \Pipq) \bx \|^2 & = & 1\\
\| \Pipq \bx \|^2  & = & 1 - \|(I - \Pipq) \bx \|^2 \\
\| \Pipq \bx \|^2 & \geq & {1 - \frac{\epsilon^2}{\delta_{p,q}(\mu)^2}} \\
\| \Pipq \bx \| & \geq & \sqrt{1 - \frac{\epsilon^2}{\delta_{p,q}(\mu)^2}}.
\eeqas
Using the inequality $a \leq \sqrt{a}$ for $a \in (0,1)$, we see
\[
\|\bv - \Pipq \bx \|
  = 1 - \|\Pipq \bx \|  \leq  1 - \sqrt{1 - \frac{\epsilon^2}{\delta_{p,q}(\mu)^2}}
  \leq \frac{\epsilon}{\delta_{p,q}(\mu)}.
\]
Then, because $(\bx - \Pipq \bx)^t (\Pipq \bx - \bv) = 0$, we have
\beqas
\|\bx - \bv\|^2 =
\|\bx - \Pipq \bx\|^2 +
\|\Pipq \bx - \bv\|^2  \leq \frac{2 \epsilon^2}{\delta_{p,q}(\mu)^2}.
\eeqas


\end{proof}


\begin{rem}
Note that the size of the {\em blend gap}, $\blendgap$, is dependent on
\begin{inparaenum}[\itshape i\upshape)]
\item the size of $|\lambda_p - \lambda_{q}|$,
\item how internal \rayquot{} $\mu$ is within $\blendinterval$,
\item and how far the exterior eigenvalues are, $|\lambda_q - \lambda_{q+1}|$ and $|\lambda_{p-1} - \lambda_p|$.
\end{inparaenum}
For a problem where the spectrum is not known \apriori{}, it is difficult to state an acceptable interval $\blendinterval$ for accomplishing a given data mining task.
\Cref{sec:ROCepairs} provides an example where one can choose $\blendinterval$ \apriori.
The Congress graph has $\lambda_{41}-\lambda_{42} \ge 0.4$ and the first $41$ eigenvectors indicate the natural clustering of the graph into sessions\cite{cucuringu2011localization}. This analysis thus applies to this real world network.
\end{rem}

For our application, $p=1$ and $\delta_{pq}=\mu - \lambda_{q+1}\paren*{\ahatshifted} \ge \lambda_q - \lambda_{q+1}$,
which can be much larger than the spectral gap $\lambda_1\paren{\ahatshifted} - \lambda_2\paren{\ahatshifted}$.
In Section~\ref{sec:convergesingle} the goal of computation is a single eigenvector and the output of the approximation is a blend of eigenvectors, the coefficients of the output in the eigenbasis of the matrix describes the error introduced by approximation.
In Section~\ref{sec:convergeblends} the goal of computation is a blend of eigenvectors, and
 we improve the error bound when the spectral gap is small.

\newcommand{\blenderrorRHS}{\frac{\sqrt{2}\epsilon}{\delta{pq}(\mu)}}

\newcommand{\errs}{\bz}
\newcommand{\errsprime}{\errs^\prime}
\newcommand{\err}[1]{z_{#1}}
\renewcommand{\ones}{\mathbf{1}}
\newcommand{\means}{\boldsymbol{\mu}}
\newcommand{\mean}[1]{\mu_{#1}}
\newcommand{\tol}{\epsilon}
\newcommand{\blendspace}{Span\{\bv_1\dots \bv_q\}}
\renewcommand{\half}{\frac{1}{2}}
\newcommand{\block}[1]{{\cal{B}}_{#1}}
\newcommand{\minperb}{\sqrt{2qn}}
\newcommand{\minperbsq}{{2qn}}
\newcommand{\minperbfrac}{\paren{1+2qn}^{-\half}}
\newcommand{\blockconstspace}{{\cal{W}}}
\newcommand{\bI}{{\mathbf{I}}}
\newcommand{\Proj}[1]{{\mathbf{P}_{#1}}}
\newcommand{\indicator}[2]{\be_{#1}^{#2}}
\newcommand{\minelt}[1]{\min_{i}{#1}_i}
\newcommand{\maxelt}[1]{\max_{i}{#1}_i}
\newcommand{\blockindicator}[2]{\be_{\block{#1}}}
\newcommand{\minperbupperbound}{b^{-\half}}
\newcommand{\robustness}[2]{g_{#1}\paren{#2}}
\newcommand{\gv}{g_\bv}
\newcommand{\gvt}{\robustness{\bv}{t}}

In order to relate the numerical accuracy to the conductance for general graphs we examine the impact of pointwise error on sweep cuts.
For any prescribed conductance value $\psi$, we derive a condition
on vectors $\bv$ such that we can guarantee that small perturbations of $\bv$ have conductance less than or equal to $\psi$.
Let $\sweepcut{\bv}{t}$ represent the sweep cut of $\bv$ at $t$ as in \cref{eqn:sweepcut}.
\begin{lemma}\label{thm:generallinftybound}
For any graph $G$, vector $\bv\in\R^n$ and scalar $\psi>0$,
define $T_{\psi}\paren{\bv}=\set{t\mid \conductance{S_\bv(t)} \le \psi}$.
Let $\gvt = \min_i \abs{v_i-t}$ and $\gv = \max_{t\in T_{\psi}\paren{\bv}} \gvt$.
If $\norm{\bz}_\infty < \gv$, then $\conductance{\bv+\bz} \le \psi$.
\end{lemma}
\begin{proof}
If $S_\bv(t) = S_{\bv+\bz}(t)$ we can apply 
$\conductance{\bv+\bz} 
    < \conductance{\sweepcut{\bv+\bz}{t}}
    = \conductance{\sweepcut{\bv}{t}} 
    \le \psi$. 
$\sweepcut{\bv}{t} = \sweepcut{\bv+\bz}{t}$ if and only if $sign(v_i+z_i-t) = sign(v_i-t)$ for all $i$.
By checking cases, one can see that $\norm{\bz}_\infty < \min_i \abs{v_i-t}$ 
is sufficient to guarantee that $v_i+z_i-t$ has the same sign as $v_i-t$.
\end{proof}
Note that \cref{thm:generallinftybound} is not a necessary condition as $\bz=\half \paren{\bv-t\ones}$ is a much larger
perturbation of $\bv$ such that $\conductance{\bv+\bz} \le \psi$.
\Cref{thm:generallinftybound} defines $\gv$ as a measure of sensitivity of a single vector with respect to preserving
sweep cuts of conductance less than or equal to $\psi$.
For vectors $\bv$ with small $\gv$, a small perturbation can disrupt the sweep cuts which achieve conductance less than
$\psi$.
By defining the sensitivity of an invariant subspace appropriately,
\cref{thm:blendspaceconductance} provides a path to deriving
a residual tolerance for arbitrary graphs.
Denote by $d_{min}, d_{max}$ the minimum and maximum degree of $G$.
\newcommand{\dminmax}{\sqrt{\frac{d_{min}}{d_{max}}}}
\newcommand{\dmaxmin}{\sqrt{\frac{d_{max}}{d_{min}}}}
\begin{thm}\label{thm:blendspaceconductance}
    Let $G$ be a graph and $\psi > 0$. 
    Define $V=\spann{\set{\Dneghalf \bv_1 \dots \Dneghalf\bv_q}}$, where for $j\in\set{1\dots q}$ $\Ah\bv_j = \lambda_q \bv_j$ and $\bv_j$ are orthogonal.
    For any vector $\bx$, let $\mu = \bx^t\Ah\bx$ and $\blendgap=\blendgapdef$. 
    For any $\bq \in V$, let $\gv$ be defined as in \cref{thm:generallinftybound}.
    Define $g=\min_{\bv\in V, \norm{\bv}_2 = 1} g_\bv$.
    If  $\norm{\Ah\bx - \mu\bx} < \frac{1}{\rtwo}\dminmax\blendgap g$, then $\conductance{\Dneghalf\bx} \le \psi$.
\end{thm}

\begin{proof}
    By \cref{thm:blenderror} applied to $\bx$, there is a unit vector $\bq\in \blendspace$ such that $\norm{\bx -\bq}_\infty \le \norm{\bx-\bq}_2 < \dminmax g$.
    Define
    $\bz = \paren{\Dneghalf\bx - \bv}\norm{\bv}^{-1}$, where $\bv = \Dneghalf\bq\in V$.
    By scaling and normalizing we see 
    $$\norm{\bz}_2 = \frac{\norm{\Dneghalf\bx - \Dneghalf\bq}_2}{\norm{\Dneghalf\bq}_2} < \sqrt{\frac{d_{max}}{d_{min}}}\norm{\bx-\bq} < g $$
    Since $$\norm{\bz}_\infty < g = \min_{\bv\in V, \norm{\bv}_2 = 1} \max_{t\in T_\bv}\min_{i\in\set{1\dots n}}\abs{v_i-t} < \gv,$$
    \cref{thm:generallinftybound} implies $\conductance{\Dneghalf\bx} \leq \psi$.
\end{proof}
If one can bound the value of $g$ from below, then this theorem gives a residual tolerance for the eigenvector approximation when using sweep cuts to partition the graph.
\Cref{sec:rocbygeneral} applies this theorem to the \roc{} family of graphs.

This section connects the eigenresidual to the error when computing blends of eigenvectors, and quantifies the impact of error on the sweep cut procedure.
 If the eigengap is small enough, then one cannot guarantee that the \rayquot{} is closest to the Fielder value,
 thus one cannot guarantee that the computed eigenvector is close to the desired eigenvector.
 In this small gap setting, a small \eres{} indicates that the computed vector is close to the desired invariant
 subspace.
 \Cref{thm:blendspaceconductance} shows that vectors with small eigenresidual preserve low conductance sweep cuts for general
 graphs. 
 \Cref{thm:blendspaceconductance} illustrates how the residual tolerance depends on both the blend gap $\blendgap$ and the sensitivity $g$ of the
 eigenvectors. 
%
The following section applies this theory to the \roc{} in order to derive solver tolerances for
 graph partitioning.
\newcommand{\corners}{corners}
\newcommand{\nonlocalized}{non-localized}
\newcommand{\intrange}[2]{#1,\dots,#2}
\newcommand{\cosvec}[2]{\bc^{#1,#2}}
\newcommand{\sinvec}[2]{\bs^{#1,#2}}
\newcommand{\pervec}[2]{\bp^{#1,#2}}
\newcommand{\lambdanoise}[1]{-(b-1)^{-1}}
\newcommand{\rocnumglobalnegevals}{q-1}
\newcommand{\LSEV}{LSEV}
\newcommand{\LSEVs}{LSEVs}
\newcommand{\rocdiagblock}{J_b}
\newcommand{\rocdiagblockdef}{\bfo_b \bfo_b^t - I_b}
\newcommand{\rocoffdiagblock}{\be_1 \be_1^t}
\newcommand{\bchoosetwo}{{b\choose 2}}
\renewcommand{\fiedvec}{\bv_2}
\renewcommand{\fiedval}{\lambda_2}
\newcommand{\blendspacemat}{\cX_{signal}}
\newcommand{\xnaught}{\bx^{(0)}}
\newcommand{\clique}[1]{{\cal K}_{#1}}
\newcommand{\sigv}[2]{\bv_{#1}^{\paren{{#2}}}}
\newcommand{\siglambda}[2]{\bv_{#1}^{\paren{{#2}}}}
\newcommand{\sigxie}[2]{\bv_{#1}^{\paren{{#2}}}}
\newcommand{\sigbeta}[2]{\bv_{#1}^{\paren{{#2}}}}
\newcommand{\sigalpha}[2]{\bv_{#1}^{\paren{{#2}}}}
\newcommand{\roceigenvectorsplotroc}{{\cal R}_{16,10}}
\newcommand{\mistakefree}{mistake-free}
\newcommand{\Mistakefree}{Mistake-free}
\newcommand{\tradeoff}{trade-off}
\newcommand{\blendwidth}{blendwidth}
\newcommand{\bQ}{\mathbf{Q}}

\section{The Ring of Cliques}\label{sec:ROCstart}

To demonstrate the theory developed in \cref{sec:blends} and to make our arguments as clear as possible,
we employ a previously studied model problem, the \roc{}~\cite{Fortunato2007resolutionlimit}.  
Theorems~\ref{thm:Xnoise}~and~\ref{thm:Epairs} derive explicit
formulas for all eigenvalues and eigenvectors. These formulas determine the relevant residual tolerance.
Moreover, complete spectral knowledge gives a strong understanding the convergence properties of simple
iterative methods.   

 The \roc~has several attractive properties for analysis of spectral partitioning.
 The Community structure is as extreme as possible for a connected graph, so the solution is well-defined.
 Also, we can apply theorems about block circulant matrices~\cite{tee2005eigenvectors} to produce closed form solutions to the eigenvector problem.
 This graph serves as a canonical example of when solving an eigenproblem accurately is unnecessarily expensive to achieve data analysis success. This example shows that it is possible for the
 combinatorial structure of the data to be revealed faster than the algebraic structure of the associated matrices.
 The graph is simple to partition accurately as there are many cuts relatively close to the minimum.
 Any robust partitioning algorithm will correctly recover the cliques in this graph.
 However, a Fiedler eigenvector is difficult to calculate with guarantees of
  point-wise accuracy when using non-preconditioned iterative methods.
 An algorithm that computes a highly accurate eigenpair will be inefficient on large problem instances.
 \Cref{sec:rocbygeneral,sec:roc-perturbation} apply the tools from \Cref{sec:blends} in order to derive a residual tolerance sufficient for solving the \roc{}.
 \Cref{sec:powermethod} bounds the number of power method iterations necessary to recover the \roc{}, and
 \Cref{sec:roc_experiment} validates and illustrates these observations with an experiment.

 
\subsection{Definition}\label{sec:ROCdefinition}
A $q$-ring of $b$-cliques, $\Roc$, is parameterized by a block size $b$ and a number of blocks $q$.
Each block represents a clique of size $b$ and all possible internal connections exist within each individual set of $b$ vertices.
For each block, there is a single vertex called the {\it{corner}}\/ connected to the corners of the adjacent cliques.
These $q$ corners form a ring.
Each block also has  $\paren{b-1}$ {\it internal}\/ vertices that have no edges leaving the block.
The adjacency matrix associated with $\Roc$ is a sum of tensor products of simple matrices (identity, cycle, and rank-one matrices).
We have
$$
	A = I_q \otimes (\rocdiagblockdef) + C_q \otimes (\rocoffdiagblock),  
$$
where the $I_k$ are identity matrices of dimension $k$, $\bfo_b$ is a constant vector with dimension $b$, $\be_1 \in
\mathbb{R}^b$ is the cardinal vector with a one in its first entry and zero elsewhere, and $C_q$ is the adjacency matrix
of a cycle graph on $q$ vertices.
The matrix $A$ and other matrices associated with this graph are {\em block-circulant}, which implies the eigenvectors are the Kronecker product of a periodic vector and the eigenvectors of a small eigenproblem defined by structure in the blocks.
\Cref{fig:rocvizgraph} shows the structure of the graph, and \Cref{fig:rocvizmat} shows the block structure of the
adjacency matrix.
\begin{figure}[htb]
	\centering
    \begin{minipage}[b]{0.45\textwidth}
        \includegraphics[width=\textwidth]{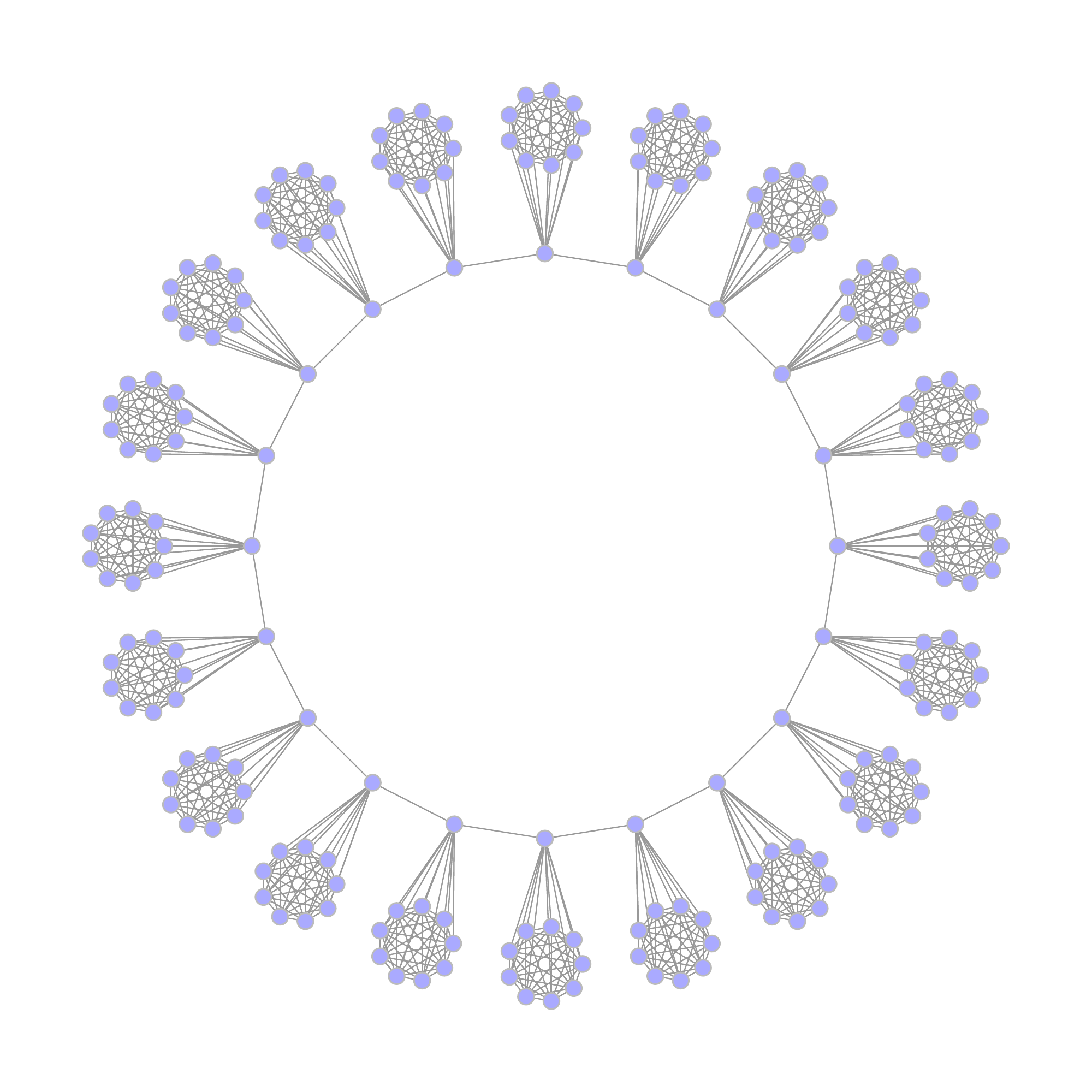}
	    \caption{A drawing of $\Roc$ laid out to show strucure.}
        \label{fig:rocvizgraph}
    \end{minipage}\quad
    \begin{minipage}[b]{0.45\linewidth}
    \setlength{\arraycolsep}{3pt}
    \[
\begin{bmatrix}
\rocdiagblock & \rocoffdiagblock & 0 & \cdots & & \rocoffdiagblock \\
\rocoffdiagblock & \rocdiagblock & \rocoffdiagblock & 0 & & \\
0 & \rocoffdiagblock & \rocdiagblock & \rocoffdiagblock & 0 & \\
\vdots && \ddots & \ddots & \ddots \\
&& 0 & \rocoffdiagblock & \rocdiagblock & \rocoffdiagblock \\
\rocoffdiagblock &&& 0 & \rocoffdiagblock & \rocdiagblock \\
 
\end{bmatrix}
    \] \\
$$
\mbox{where} \quad \rocdiagblock = \rocdiagblockdef .
$$    
    \vspace{0.5cm}
    \caption{The adjacency matrix of $\Roc{}$ has block circulant structure.}
    \label{fig:rocvizmat}
    \end{minipage}
\end{figure}

Any partition that breaks a clique cuts at least $b-2$ edges while any partition that does not break any cliques cuts at most $q$ edges.
The best partition is break the ring into two contiguous halves by cutting exactly two edges.
There are $q/2$ partitions that achieve this minimal cut for even $q$.
We will consider any of these equivalent.
Any partition that breaks fewer than $b-2$ edges will be regarded as a good, but not optimal cut.
The fact that many partitions are close to optimal and then the vast majority of partitions are very far from optimal is a feature of this model problem.

\subsection{Eigenpairs of ROC normalized Adjacency Matrix}\label{sec:ROCepairs}
\begin{figure}[h]
	\label{fig:rocspect}
	\centering
	\includegraphics[width=1.0\textwidth]{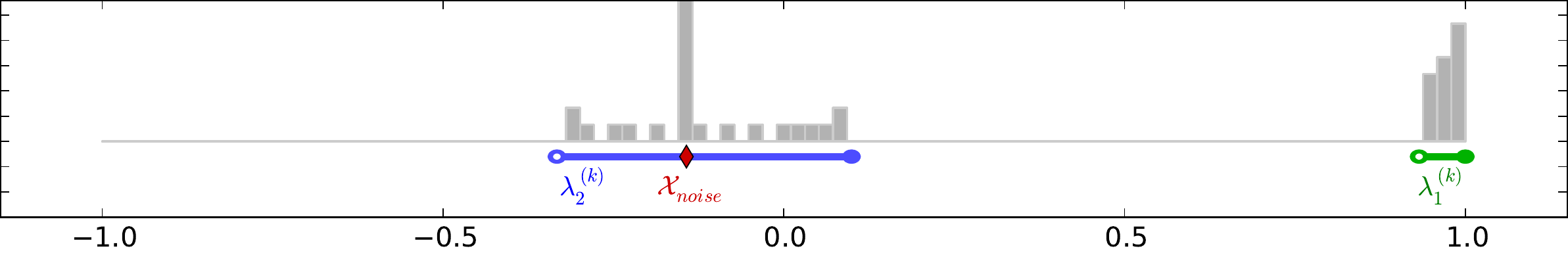}
	\caption{Distribution of eigenvalues of $\Ah$ for ${\cal R}_{b=8, q=32}$.
		The gray bars represent a histogram of the eigenvalues, including multiplcities.
		The red diamond represents a large multiplicity of $(n-2q)$ at $-(b-1)^{-1}$ corresponding to $\cX_{noise}$ (see Theorem~\ref{thm:Xnoise}).   
		There is a green interval near 1 containing the portion of the spectrum given by $\lambda^{(k)}_1$ and a blue interval near 0 containing $\lambda^{(k)}_1$ for $k=0, 1, ..., \lfloor q/2\rfloor$ (see Theorem~\ref{thm:Epairs}).
		The open left endpoint of each interval signifies that the eigenvalue corresponding to $\lambda^{(k)}_{\lceil q/2\rceil}$ is not present when $q$ is odd.}
\end{figure}

Due to the block-circulant structure of the \roc{} we are able to compute the eigenvalues and  eigenvectors in a closed form.
Let $\cosvec{k}{q},\sinvec{k}{q}$ be the periodic vectors as defined in \cref{eqn:periodicdefs}.
\begin{equation}\label{eqn:periodicdefs}
 \cosvec{k}{q}_j = \cos\paren*{\frac{2\pi k}{q}j}\quad
 \sinvec{k}{q}_j = \sin\paren*{\frac{2\pi k}{q}j}\quad \text{for}\ j=\intrange{1}{q} 
\end{equation}
These vectors are periodic functions with $k$ cycles through the interval $\clintervalcl{-1}{1}$ sampled at the $q$th roots of unity.
We employ results from \cite{Henson:2012:LSE} and \cite{tee2005eigenvectors} to derive the full spectrum of $\hat{A}$ and a full eigenbasis.
In summary, there is an eigenvalue $\lambdanoise{b}$ with a large multiplicity, $n - 2q = (b-2)q$.
Furthermore, the eigenspace associated with $\lambdanoise{b}$ can be represented in a basis that contains variation
internal to each clique, that is with eigenvectors of the form $\bh \otimes \be_i$ where $\be_i$ is the $i$th standard
basis vector for each $i\in \set{1\dots q}$.
For this reason, we call $\lambda_{noise} = \lambdanoise{b}$ a noise eigenvalue.
The positive eigenvalues are called signal eigenvalues.
The signal eigenvectors have the form $\pervec{k}{q} \otimes (\xi \be_1 + \bg_1)$, where $\pervec{k}{q}$ is either $\sinvec{k}{q}$ or $\cosvec{k}{q}$, $\be_1$ is one in its first entry, $\bg_1$ is zero in its first entry and one elsewhere, and $\xi$ is a  scalar. 
All of the internal members of the cliques take the same value in any eigenvector associated with $\lambda_k(\hat{A}) \neq (b-1)^{-1}$. 
The slowly varying eigenvectors (associated with $\lambda_k(\hat{A}) \approx 1$) give nearly optimal
partitions of the graph.
Linear combinations of these slowly varying signal eigenvectors also give low conductance partitions.
There are $\rocnumglobalnegevals$ \nonlocalized{} eigenvectors with small positive or negative eigenvalues.
These eigenvectors have the internal clique members and their corner with different sign which causes them to
misclassify some of the corners.
The distribution of the eigenvalues of $\Roc$ is illustrated in \Cref{fig:roceigenvectorsplot}.
The rest of \cref{sec:ROCepairs} contains formulas for the eigenpairs and the details of their derivations.
\begin{thm}{\bf ($\Roc$ Noise Eigenpairs)}
\label{thm:Xnoise}
There is an eigenspace ${\cal X}_{noise}$ of multiplicity $(n - 2q)$ associated with eigenvalue
\[
\lambda_{noise} = \frac{-1}{b-1}.
\]
Any vector that is zero-valued on all corner vertices and sums to zero on each individual set of internal vertices is in ${\cal X}_{noise}$.
\end{thm}

\begin{proof}
    This is a specific case of locally supported eigenvectors~\cite{Henson:2012:LSE} (\LSEVs{}) brought on by a high level of local symmetry in $\Roc$. 
    For each clique $\clique{}$, let $\bx = \be_i - \be_j$ for vertices $i$ and $j$ that are internal vertices of $\clique{}$.
    Both $\bx$ and $\Ah \bx$ are zero valued outside of $\clique{}$.
    Internally, due to $D^{-1/2} (\be_i - \be_j) = (b-1)^{-1/2}(\be_i - \be_j)$ and the orthogonality $\bfo_b^t (\be_i - \be_j) = 0$, we see
\[
\Ah (\be_i - \be_j) =
\frac{1}{b-1} (\bfo_b \bfo_b^t - I ) (\be_i - \be_j) = \frac{-1}{b-1} (\be_i - \be_j). 
\]
Thus, $\bx$ is an eigenvector of $\Ah$ associated with $-1/(b-1)$.
There are $(b-2)$ such linearly independent eigenvectors for $\clique{}$, and the same is true for all $q$ cliques.
Thus, we have a multiplicity of $q(b-2) = n - 2q$ for eigenvalue $\lambda_{noise} = (b-1)^{-1}$.

\end{proof}

These vectors are in the interior of the spectrum and thus are very well attenuated by the power-method\footnote{In a single iteration the shifted power method,  $\bx_{k+1} = (\hat{A} + (b-1)^{-1} I) \bx_{k}$,  perfectly eliminates all of the energy in the $(n-2q)$-dimensional eigenspace associated with $\lambda=-(b-1)^{-1}$. If the graph is perturbed with the addition and removal of a few edges, the eigenvectors become slightly less localized and the associated eigenvalues spread out to a short range of values and are not perfectly eliminated in a single iteration.   However, the power method or a Krylov method will rapidly attenuate the energy in the associated eigenspaces.}.
The remaining eigenvectors must be constant on the internal nodes of the blocks because of orthogonality to the \LSEVs{}
which are spanned by $\be_i-\be_j$.
In any vector $\bv$ the projection of $\bv$ onto the global eigenvectors defines a mean value for the elements of the
blocks. Since all of the eigenvectors of interest are orthogonal to the constant vector, their entries must sum to zero.
So the \LSEVs{} cannot change the mean of a block. The remaining eigenvectors are given in \cref{thm:Epairs}.




\begin{thm}{\bf (ROC Signal Eigenpairs)} For $k=0,..., \lceil \frac{q}{2} \rceil - 1$, define
\label{thm:Epairs}
\begin{eqnarray*}
\alpha_k & = & 2 \cos \left( \frac{2 \pi k}{q}\right) \\
\beta_k & = & \frac{1}{2} \left( \alpha_k \sqrt{\frac{b-1}{b+1}} - \sqrt{b^2-1} + \sqrt{\frac{b+1}{b-1}} \right) \\
\xi_1^{(k)} & = & \beta_k + \sqrt{\beta_k^2+(b-1)} \\
\xi_2^{(k)} & = & \beta_k - \sqrt{\beta_k^2+(b-1)} \\
\lambda_1^{(k)} & = & \frac{\xi^{(k)}_1}{\sqrt{b^2-1}} + 1 - \frac{1}{b-1} \\
\lambda_2^{(k)} & = & \frac{\xi^{(k)}_2}{\sqrt{b^2-1}} + 1 - \frac{1}{b-1} \\
\end{eqnarray*}
\begin{itemize}

Let $\bfo_b$ and $\bfo_q$ be the vectors of all ones in $\mathbb{R}^b$ and $\mathbb{R}^q$, respectively.  Also let $\be_1 \in \mathbb{R}^b$ have a one in its first entry, zero elsewhere and $\bg_1 = \bfo_b - \be_1$.   We have the following eigenpairs.

\item[(i)] For $k=0$, we have 2 eigenvalues of $\hat{A}$, $\lambda_1^{(0)}$ and $\lambda_2^{(0)}$, each with multiplicity 1.    The associated (unnormalized) eigenvectors are 
\[
\bv_1^{(0)} = \sqrt{b+1} (\bfo_q  \otimes \be_1) + \sqrt{b-1} (\bfo_q  \otimes \bg_1) 
\]
and
\[
\bv_2^{(0)} = (b-1)^{3/2} (\bfo_q  \otimes \be_1) - \sqrt{b+1} (\bfo_q  \otimes \bg_1) 
\]
respectively. 

\item[(ii)] For each $k=1, ..., \lceil \frac{q}{2} \rceil - 1$, we have 2 eigenvalues of $\hat{A}$, $\lambda_1^{(k)}$ and $\lambda_2^{(k)}$, each with multiplicity 2.
Two independent (unnormalized) eigenvectors associated with $\lambda_1^{(k)}$ are
\[
\bv_{1,1}^{(k)} = \cosvec{k}{q} \otimes \left(\xi_1^{(k)}\be_1 + \bg_1 \right)\label{eqn:kroneker-signalvectors}
\quad \mbox{and} \quad
\bv_{1,2}^{(k)} = \sinvec{k}{q} \otimes \left(\xi_1^{(k)}\be_1 + \bg_1 \right). 
\]
Two independent (unnormalized) eigenvectors associated with $\lambda_2^{(k)}$ are
\[
\bv_{2,1}^{(k)} = \cosvec{k}{q} \otimes \left(\xi_2^{(k)}\be_1 + \bg_1 \right)
\quad \mbox{and} \quad
\bv_{2,2}^{(k)} = \sinvec{k}{q} \otimes \left(\xi_2^{(k)}\be_1 + \bg_1 \right).
\]
\item[(iii)] If $q$ is even, then for $k=\frac{q}{2}$, we have 2 eigenvalues of $\hat{A}$, $\lambda^{(q/2)}_1$ and $\lambda^{(q/2)}_2$, each with multiplicity 1.   The associated (unnormalized) eigenvectors are 
\[
\bv_1^{(q/2)} = \cosvec{q/2}{q} \otimes \left(\xi_1^{(q/2)}\be_1 + \bg_1 \right)
\]
and
\[
\bv_2^{(q/2)} = \sinvec{q/2}{q} \otimes \left(\xi_2^{(q/2)}\be_1 + \bg_1 \right)
\]
respectively.

\end{itemize}

Note if values of $\lambda_p^{(k)}$ and $\lambda_q^{(l)}$ coincide for $(p,k)\neq(q,l)$ the eigenvalue multiplicities add up.

\end{thm}

\begin{proof} 
    Let $D_b = \mbox{diag}((b+1)\bfo - 2 \bg_1)$.   Matrix $\hat{A}$ can be organized into blocks that all
    co-commute with each other implying the eigenvectors are tensor products involving the eigenvectors of the blocks.
    We  decompose $\hat{A} = (I_q \otimes B_1) + (C_q \otimes B_2)$, where $I_q$ is the identity in $\mathbb{R}^{q}$,
    $C_q$ is the adjacency matrix of a $q$- cycle, $(C_q)_{ij} = 1 \Longleftrightarrow |i-j|=1$ mod $q$, $B_1 =
    D_b^{-1/2}(\bfo_b \bfo_b^t-I)D_b^{-1/2}$, and $B_2 = \frac{1}{b+1}\be_1 \be_1^t$.  Because any eigenvector $\by$ of
    $C_q$ is also an eigenvector of $I_q$, eigenvectors of $\hat{A}$ have the form $\by \otimes \bz$.      Vectors $\bz$
    are derived by plugging various eigenvectors of $C_q$ into $\by$ and solving for a set of constraints that $\bz$
    must satisfy for $(\by \otimes \bz)$ to be an eigenvector associated with $\hat{A}$.

We describe the eigendecomposition of $C_q$.   For $k=0, \dots, \lceil \frac{q}{2}\rceil-1$, $\alpha_k = 2 \cos (2 \pi / k)$ is an eigenvalue of $C_q$.    For $k=0$, $\alpha_0$ is simple, and $\bfo_b$ is the associated eigenvector.   For $k=1, ..., \lceil \frac{q}{2}\rceil$-1, $\alpha_k$ has multiplicity 2 and the associated $2$-dimensional eigenspace is span($\{\cosvec{k}{q},\, \sinvec{k}{q}$\}), as defined in (\ref{eqn:periodicdefs}).
If $q$ is even, then $\alpha_{q/2}$ is a simple eigenvalue as well and the associated eigenvector is $\cosvec{q/2}{q}$.
Letting $\by$ be an eigenvector associated with $\alpha_k$ and using the properties of Kronecker products, we see
\beqas
\hat{A} (\by \otimes \bz) & = & 
\left[ (I_q \otimes B_1) + (C_q \otimes B_2) \right](\by \otimes \bz)  
= \left[ (I_q \by \otimes B_1 \bz) + (C_q \by \otimes B_2 \bz) \right] \\
& =& \left[ (\by \otimes B_1 \bz) + ( \alpha_k \by \otimes B_2 \bz) \right] =
 \left[ (\by \otimes B_1 \bz) + ( \by \otimes \alpha_k B_2 \bz) \right]\\
& =& \by \otimes \left(B_1 \bz +\alpha_k B_2 \bz \right)  =
\by \otimes  \left[ (B_1 +\alpha_k B_2 ) \bz\right] \\
\eeqas
Here we see that if $\bz$ is an eigenvector of $H_k := B_1 + \alpha_k B_2$, then $(\by \otimes \bz)$ is an eigenvector associated with $\hat{A}$.   Observe that $H_k = D_b^{-1/2} (\bfo_b \bfo_b^t +\alpha_k \be_1 \be_1^t - I) D_b^{-1/2} $ is a scaling of a rank-2 shift from the identity matrix, where we would expect 3 eigenvalues: 2 simple and one of multiplicity $(b-2)$.

We can easily verify that there is a $(b-2)$-dimensional eigenspace of $H_k$ associated with $-1/(b-1)$.    The tensor products of these vectors are an alternative basis associated with the locally supported eigenvectors from Theorem~\ref{thm:Xnoise}.   The associated eigenspace of $H_k$ is orthogonal to span($\{\be_1, \bg_1\}$).   Due to eigenvector orthogonality, the last two eigenvectors must be in the range of span($\{\be_1, \bg_1\}$).  Note that $D_b^p \be_1=  (b+1)^p \be_1$ and $D_b^p \bg_1=  (b-1)^p \bg_1$.   We use this to solve for these eigenvectors and their associated eigenvalues in terms of $\alpha_k$ and $b$,
\beqas
H_k ( \xi \be_1 + \bg_1) & = & \lambda ( \xi \be_1 + \bg_1) \\
(\bfo_b \bfo_b^t + \alpha_k \be_1 \be_1^t - I_b )  D_b^{-1/2}( \xi \be_1 +  \bg_1) & = & \lambda  D_b^{1/2}  (\xi \be_1 + \bg_1)
\eeqas
The right-hand side expands to
$$
\left( \lambda \xi \sqrt{b+1} \right) \be_1 + \left( \lambda \sqrt{b-1} \right) \bg_1,
$$
The left-hand side expands and simplifies to
\beqas
 \left( \frac{\xi}{\sqrt{b+1}} + \sqrt{b-1} \right) (\be_1 + \bg_1) +  \left(  \frac{\xi \alpha_k }{\sqrt{b+1}} - \frac{\xi}{\sqrt{b+1}}  \right) \be_1+  \left( \frac{-1}{\sqrt{b-1}}  \right) \bg_1  & = & \\
\left(  \frac{\xi \alpha_k}{\sqrt{b+1}} + \sqrt{b-1}  \right) \be_1+  \left(  \frac{\xi }{\sqrt{b+1}} + \sqrt{b-1} - \frac{1}{\sqrt{b-1}}   \right) \bg_1  & = & \\
\eeqas

The coefficients of $\be_1$ and $\bg_1$ must be equal, individually, because they are not linearly dependent.   Equating the left-hand and right-hand sides and simplifying gives two nonlinear equations in $\xi$ and $\lambda$,
\begin{eqnarray}
 \frac{\xi \alpha_k}{b+1} + \sqrt{\frac{b-1}{b+1}} & = & \xi \lambda \label{eqn:nleqs1}\\
  \frac{\xi }{\sqrt{b^2-1}} + 1 - \frac{1}{b-1}  & = & \lambda. \label{eqn:nleqs2}
\end{eqnarray}
Multiplying the second equation by $\xi$, setting the left-hand sides of both equations equal to eliminate $\lambda$, then simplifying, yields the following quadratic equation in $\xi$,
$$
\xi^2 - \left( \alpha_k \sqrt{\frac{b-1}{b+1}} - \sqrt{b^2 - 1} + \sqrt{\frac{b+1}{b-1}}\right)\xi - \left( b-1\right) = 0,
$$
which is easily solved.   Define
$$
\beta_k = \frac{1}{2} \left(\alpha_k \sqrt{\frac{b-1}{b+1}} - \sqrt{b^2 - 1} + \sqrt{\frac{b+1}{b-1}} \right)\qq{and}
\gamma_k = b-1.
$$
Given $\xi$, $\lambda$ is determined by the second equation in (\ref{eqn:nleqs2}).   The solution set to nonlinear equations~(\ref{eqn:nleqs1})-(\ref{eqn:nleqs2}) is then
\beqas
\xi^{(k)}_1 = \beta_k + \sqrt{\beta_k^2 + \gamma_k},
&&
\lambda^{(k)}_1 = \frac{\xi^{(k)}_1 }{\sqrt{b^2-1}} + 1 - \frac{1}{b-1} , \qq{and}\\
\xi^{(k)}_2 = \beta_k - \sqrt{\beta_k^2 + \gamma_k}, 
&& 
\lambda^{(k)}_2 = \frac{\xi^{(k)}_2 }{\sqrt{b^2-1}} + 1 - \frac{1}{b-1}.
\eeqas

Thus we have local eigenpairs of $H_k$, $((\xi^{(k)}_1 \be_1 + \bg_1), \lambda_1^{(k)})$ and $( (\xi^{(k)}_2 \be_1 + \bg_1), \lambda_2^{(k)})$.   
The local eigenpairs  $((\xi^{(k)}_1 \be_1 + \bg_1), \lambda_1^{(k)})$ 
yield global eigenpairs of $\hat{A}$ of the form 
$( (\bc_k \otimes (\xi^{(k)}_1 \be_1 + \bg_1) ), \, \lambda_1^{(k)})$
and $( (\bs_k \otimes (\xi^{(k)}_1 \be_1 + \bg_1) ), \, \lambda_1^{(k)})$.
Similarly, $((\xi^{(k)}_1 \be_1 + \bg_1), \, \lambda_2^{(k)})$ yield global eigenvectors of $\Ah$ associated with
 $\lambda_2^{(k)}$.
This accounts for the last $2q$ eigenpairs of $\hat{A}$.
\end{proof}

In order to illustrate these formulas, \Cref{fig:roceigenvectorsplot} shows
the computed eigenvectors for the graph $\roceigenvectorsplotroc$ along with the eigenvalues.
Eigenvectors associated with eigenvalues close to $1$ have low conductance sweep cuts.
\begin{figure}[h]
    \includegraphics[width=1.0\textwidth]{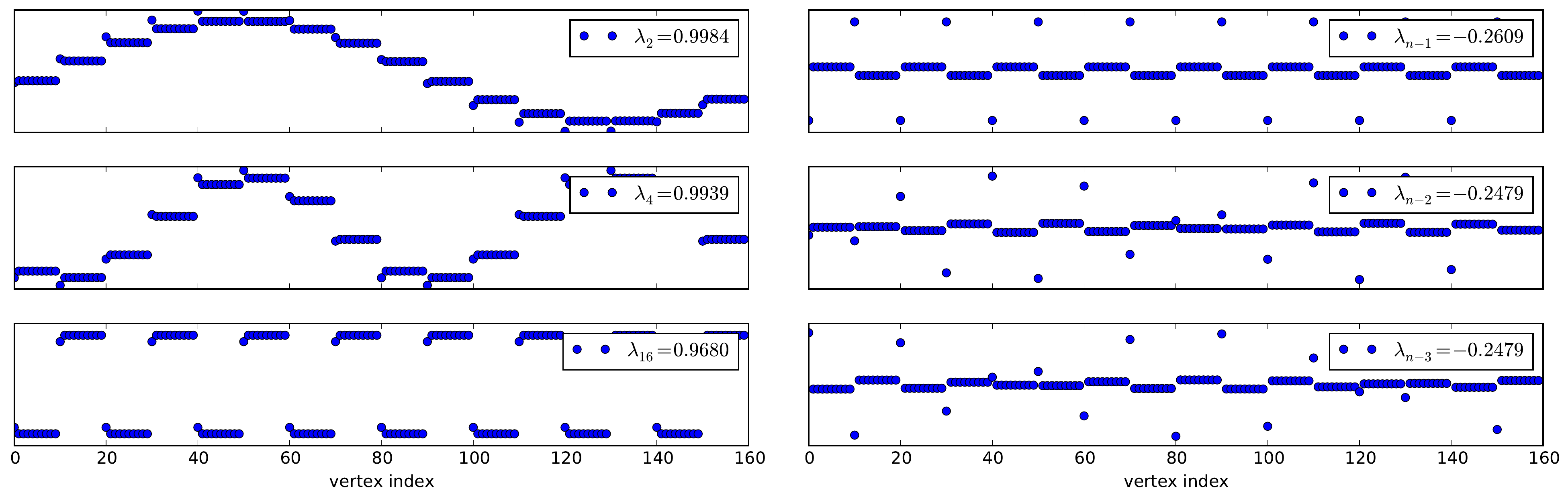}
    \caption{The eigenvectors of $\Ah$ are shown for $\roceigenvectorsplotroc$. The eigenvectors with eigenvalues close to $1$
    (left) indicate the block structure with differing frequencies. The eigenvectors close to $-1$ (right) assign
    opposite signs to the internal vertices and corner vertices of each block.}
    \label{fig:roceigenvectorsplot}
\end{figure}

\begin{coro}
\label{cor:evals}
The asymptotic expansions of the eigenvalues are as follows.   
\begin{itemize}
\item[(i)] For the signal eigenpairs, we see
$$
\lambda_1^{(k)} = 1- \frac{4-\alpha_k}{2 (b^2 - 1)} + \frac{1}{4(b-1)^2} + \frac{\alpha_k^2}{4(b+1)^2} + \cO(b^{-3})
$$
\item[(ii)] For the non-signal (and non-noise) eigenpairs, we see
$$
\lambda_2^{(k)} = \frac{\alpha_k-1}{b+1} - \frac{\alpha_k^2}{4(b+1)^2} - \frac{\alpha_k}{2 (b^2 - 1)}  -  \frac{1}{4 (b - 1)^2} + \cO(b^{-3}) 
$$
\item[(iii)] In particular, if $q$ is even we have
\begin{eqnarray*}
\lambda_2 = \lambda_1^{(1)} & \approx & 1 - \left[2 - \cos\left(\frac{2 \pi}{q}\right)\right] \frac{1}{b^2-1} +\frac{1}{4(b-1)^2} + \cos^2 \left(\frac{2 \pi}{q}\right) \frac{1}{(b+1)^2} , 
\\
\lambda_q = \lambda^{(q/2)}_{1} & \approx & 1 - \frac{3}{b^2-1} +\frac{1}{4(b-1)^2} + \frac{1}{(b+1)^2}, 
\\
\lambda_{q+1} = \lambda^{(0)}_{2} & \approx & \frac{1}{b+1}  -  \frac{1}{(b+1)^2} - \frac{1}{b^2 -1} - \frac{1}{(b-1)^2},
\\
\lambda_{n} = \lambda^{(q/2)}_{2} & \approx & \frac{-2}{b+1}  - \frac{1}{(b+1)^2}  + \frac{1}{b^2 -1} - \frac{1}{(b-1)^2}. 
\end{eqnarray*}
In other words,  
\begin{eqnarray*}
 1- \frac{C_2}{b^2} < \lambda_2 < 1, &&   1- \frac{C_q}{b^2} < \lambda_q < 1, \\ \\
0  < \lambda_{q+1} < \frac{C_{q+1}}{b}, &&   - \frac{C_n}{b} < \lambda_n < 0, \\
\end{eqnarray*}
where $C_2, C_q, C_{q+1}, $ and $C_n$ are positive quantities of order 1.  
\end{itemize}
\end{coro}
\begin{proof}
These are seen through the formulas in Theorem~\ref{thm:Epairs} and first-order Taylor expansion of $\sqrt{\beta_k^2 + (b-1)}$ about leading asymptotic term $\frac{1}{4}((b^2-1)^2)$ and simplification.   Let $\theta = \sqrt{b-1}$ and $\eta = \sqrt{b+1}$.   Then
$$
\beta_k^2 = \frac{1}{4}\left(  \frac{\alpha_k \theta}{\eta} -\theta \eta + \frac{\eta}{\theta}   \right)^2 = \frac{\alpha_k^2 \theta^2}{4 \eta^2} + \frac{\theta^2 \eta^2}{4} + \frac{\eta^2}{4\theta^2} - \frac{\alpha_k \theta^2}{2} +\frac{\alpha_k}{2} - \frac{\eta^2}{2}.
$$
The first-order Taylor expansion we employ is $\sqrt{a^2+x} = a + \frac{1}{2} a^{-1} x + \cO(a^{-3} x^2)$, which yields
\beqas
\sqrt{\beta_k^2 + \theta^2} & = & \sqrt{\left[\frac{\theta \eta}{2}\right]^2 + \left[ \left(1 - \frac{\alpha_k}{2}\right)\theta^2 - \frac{\eta^2}{2} + \frac{\alpha_k^2 \theta^2 }{4 \eta^2} + \frac{\alpha_k}{2} +\frac{\eta^2}{4 \theta^2} \right]} \\
& = &  \frac{\theta \eta}{2} + \frac{1}{\theta \eta} \left[ \left(1 - \frac{\alpha_k}{2}\right)\theta^2 - \frac{\eta^2}{2} + \frac{\alpha_k^2 \theta^2 }{4 \eta^2} + \frac{\alpha_k}{2} +\frac{\eta^2}{4 \theta^2} \right] + \cO(b^{-2}) \\ 
& = &  \frac{\theta \eta}{2} +  \left(1 - \frac{\alpha_k}{2}\right)\frac{\theta}{\eta} - \frac{\eta}{2 \theta} + \frac{\alpha_k^2 \theta }{4 \eta^3} + \frac{\alpha_k}{ 2 \theta \eta} +\frac{\eta}{4 \theta^3} + \cO(b^{-2}). \\
\eeqas
Now, we see 
\beqas
\xi^{(k)}_1 & = & \beta_k + \sqrt{\beta_k^2+\theta^2} = \frac{1}{2}\left(  \frac{\alpha_k \theta}{\eta} -\theta \eta + \frac{\eta}{\theta}   \right) + \sqrt{\beta_k+\theta^2} \\
& = & \frac{\theta}{\eta} + \frac{\alpha_k^2 \theta }{4 \eta^3} + \frac{\alpha_k}{ 2 \theta \eta} +\frac{\eta}{4 \theta^3} + \cO(b^{-2}), \qquad \mbox{and } \\
\xi^{(k)}_2 & = & \beta_k - \sqrt{\beta_k^2+\theta^2} = \frac{1}{2}\left(  \frac{\alpha_k \theta}{\eta} -\theta \eta + \frac{\eta}{\theta}   \right) - \sqrt{\beta_k+\theta^2} \\
& = & -\theta\eta + \left( \alpha_k - 1\right) \frac{\theta}{\eta} + \frac{\eta}{\theta} - \frac{\alpha_k^2 \theta }{4 \eta^3} - \frac{\alpha_k}{ 2 \theta \eta} - \frac{\eta}{4 \theta^3} + \cO(b^{-2}).
\eeqas 
Then, noting $\eta^{-1} - \theta^{-1} = -2 \eta^{-1}\theta^{-1}$, we see
\beqas
\lambda^{(k)}_1 & = & 1 - \frac{1}{\theta^2} + \frac{\xi_1^{(k)}}{\theta \eta}\\
& = & 1 - \frac{1}{\theta^2} + \frac{1}{\eta^2} + \frac{\alpha_k}{2 \theta^2 \eta^2} + \frac{1}{4 \theta^4} + \frac{\alpha_k^2}{ 4 \eta^4 }\\
& = & 1 - \frac{4 - \alpha_k}{2 \theta^2 \eta^2} + \frac{1}{4 \theta^4} + \frac{\alpha_k^2}{4 \eta^4} + \cO(b^{-3}).
\eeqas 
Substituting in for $\theta$ and $\eta$ gives the result presented in (i).   Similarly, we see (ii) via
$$
\lambda^{(k)}_2  =  1 - \frac{1}{\theta^2} + \frac{\xi_2^{(k)}}{\theta \eta} = \frac{\alpha_k-1}{\eta^2} - \frac{\alpha_k^2}{4 \eta^4} - \frac{\alpha_k}{2\theta^2 \eta^2} - \frac{1}{4 \theta^4} + \cO(b^{-3}).
$$
Lastly, (iii) is seen by plugging in for specific values of $k$ and $\alpha_k = 2 \cos \left( \frac{2 \pi k }{q} \right) $. 
 
\end{proof}

\begin{rem} We observe several facts:

\bit
\item The vector $D^{-1/2 } \sigv{1}{0}$ is the constant vector.   It causes no errors, but does not help to separate any of the cliques. 

\item The vectors $D^{-1/2} \sigv{1,1}{k}$ and $D^{-1/2} \sigv{1,2}{k}$ for $k=\intrange{1}{\ceil{q/2}-1}$ assign the same sign to the \corners{} as the internal members of each block and are associated with positive eigenvalues.   Note that we can consider all these eigenvectors as signal eigenvectors, 
$$
	\cX_{signal} = \mbox{span}\left\{\sigv{1}{1}, \sigv{1}{2}, ... , \sigv{1}{\ceil{q}-1}\right\}.
$$
Because $\cX_{signal}\perp \cX_{noise}$, all sweep cuts of vectors in $\cX_{signal}$ keep internal vertices of each clique together.

\item If $q$ is even, the vector $D^{-1/2} \sigv{2}{q/2}$ has values at the \corners{} of opposite sign to the values of the internal vertices and the sign of each corner oscillates around the ring.
This is the most oscillatory global eigenvector.

\item The vectors $D^{-1/2} \sigv{2,1}{k}$ and $D^{-1/2} \sigv{2,2}{k}$ for $k=\intrange{1}{\ceil{q/2}-1}$ assign
    opposite signs to the \corners{} and the internal members of each block. If vectors make large contributions to the blend we compute, then there is potential to misclassify several of the corner vertices.
 
\eit
\end{rem}

\begin{figure}
\includegraphics[width=3.0in]{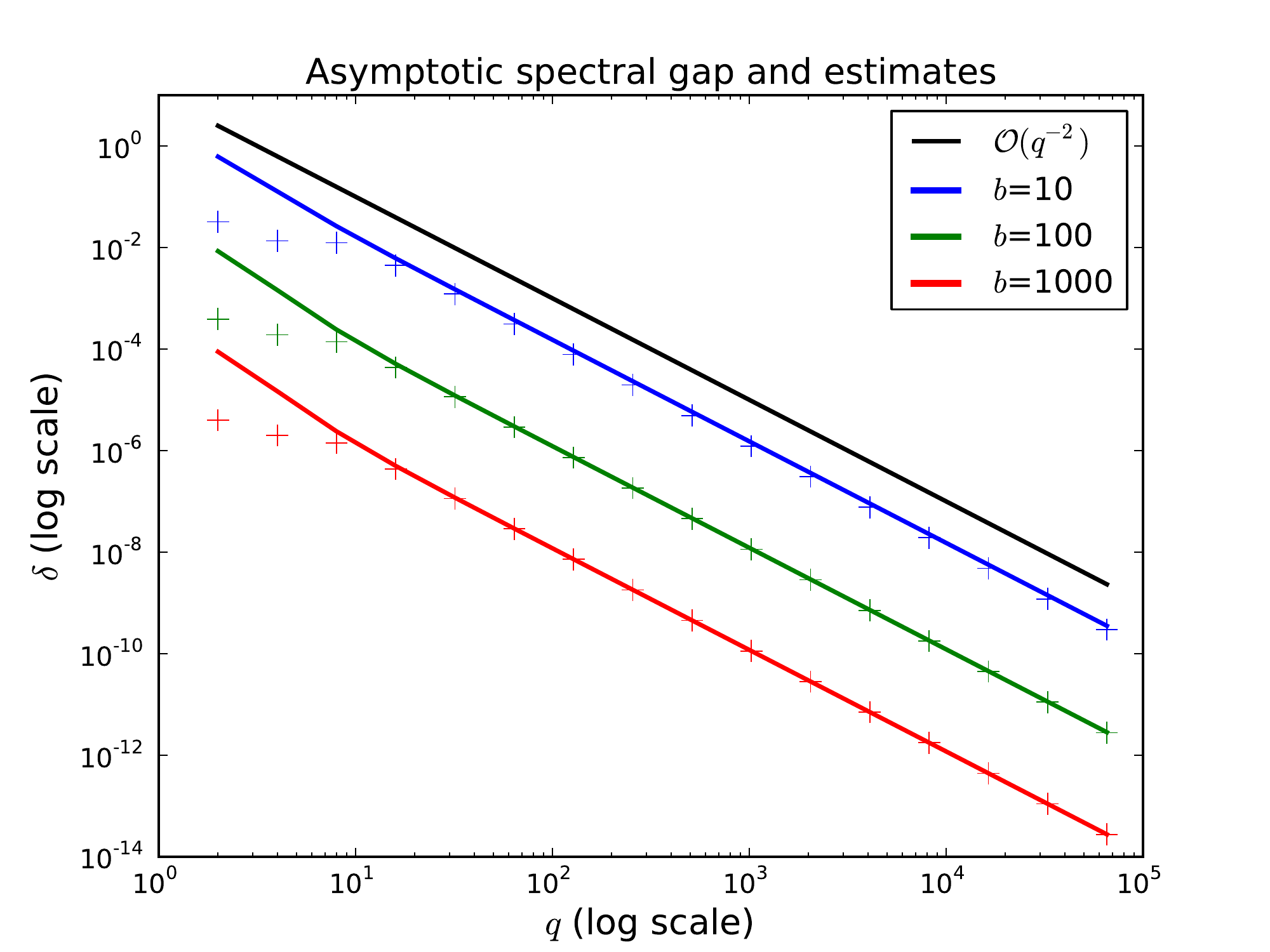}
\includegraphics[width=3.0in]{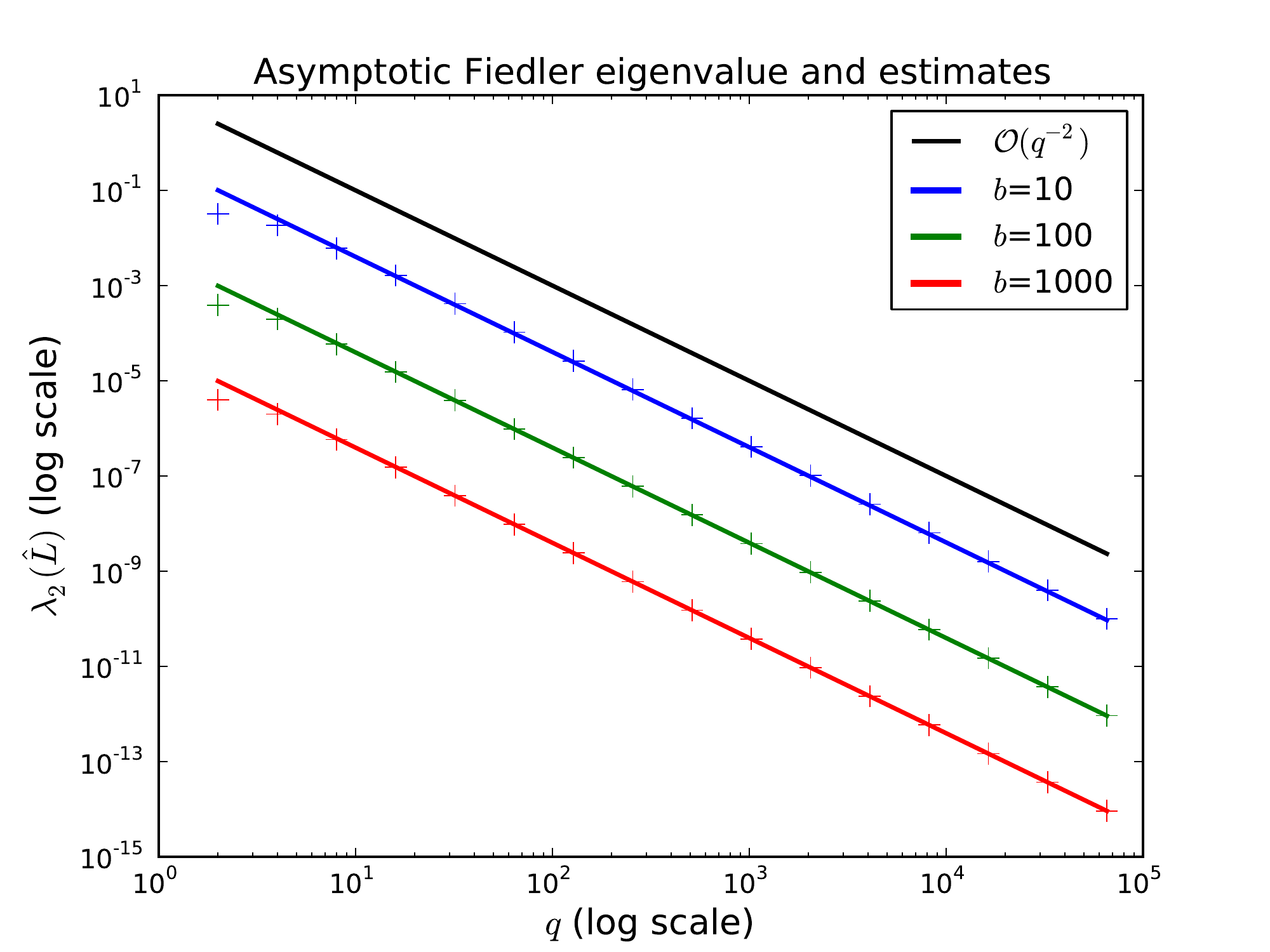}
\caption{Asymptotic estimates of spectral gaps (left) and Fiedler eigenvalues (right) for rings of cliques with parameters $b = 10,100,1000$ and $q=2,4, ..., 8096$.   Lines represent leading-order terms derived in Theorems~\ref{thm:roceigengap} and \ref{thm:fiedler} and '+' represent actual eigenvalues as given by the formulas in Theorem~\ref{thm:Epairs}. }
\label{fig:asymptotic}
\end{figure}

We use the previous result to derive asymptotic estimates of the eigengaps 
associated with eigenvalues near 1 and the size of the Fiedler eigenvalue.
These asymptotic estimates are compact formulas in terms of $b$ and $q$.
Figure~\ref{fig:asymptotic} verifies these estimates empirically.

\begin{thm} {\bf (ROC Asymptotic Eigengap)}
\label{thm:roceigengap}
For the graph $\Roc$ the spectral gap relevant for computing an eigenvector associated with $\lambda_2(\Lh)$ is asymptotically $\bigo{b^{-2} q^{-2}}$ for large $b$ and $q$.
\end{thm}

\begin{proof}
Because $\lambda_2(\Ah) = \lambda_3(\Ah)$, the eigenvalues of interest are $\lambda_2(\Ah) =  \lambda^{(1)}_1$ and
$\lambda_4(\Ah) = \lambda^{(2)}_1$.    We will take Taylor expansions and collect the leading-order terms to understand the asymptotic behavior.   Define scalar function $f(x) = \sqrt{x+a}$, for constant $a$ that we define below. Using the formulas in the previous result and a bit of algebra, we see $ \lambda^{(1)}_1 -  \lambda^{(2)}_1 = $
\beq
\label{eqn:egap_roc}
\frac{\xi^{(1)}_1 - \xi^{(2)}_1}{\sqrt{b^2 -1}} = 
 \frac{\beta_1 - \beta_2+\sqrt{\beta_1^2  + (b-1)} - \sqrt{\beta_2^2  + (b-1)} }{\sqrt{b^2 -1}} = \frac{\beta_1 - \beta_2 + f(x_1)-f(x_2)}{\sqrt{b^2 -1}}
\eeq
with 
$$
x_k = \left( \frac{1}{4} - \frac{1}{2(b+1)}\right) \alpha_k^2 + \left(1 - \frac{b}{2}\right) \alpha_k 
\qquad \mbox{and} \qquad
a = \frac{b^2}{4} + \frac{b}{2} - \frac{3}{2} + \frac{1}{2(b-1)}.
$$
We expand the two differences in the numerator of (\ref{eqn:egap_roc}) separately, concentrating on the $f(x_1) - f(x_2)$ first.   First-order Taylor expansion of $f(x)$ at $x_1$ yields
$$
f(x_2) = f(x_1) + f'(x_1)(x_2 - x_1) + \frac{f''(y)}{2} (x_2 - x_1)^2 .
$$
where $y \in (\min(x_1, x_2), \max(x_1, x_2))$.   Rearranging and plugging in, we see
$$
f(x_1) - f(x_2) = f'(x_1)(x_1 - x_2) - \frac{f''(y)}{2} (x_1 - x_2)^2 = \frac{x_1 - x_2}{2\sqrt{x_1 + a}} - \frac{(x_1 - x_2)^2}{4(y + a)^{3/2}}
$$

Further, assume $q >> 4 \pi k$ and use a  third-order Taylor expansion at 0 to see $\alpha_k = 2 - 4(\pi k/q)^2 + {\cal O}(q^{-4})$.   Similarly, $\alpha_k^2 = 4 \cos^2(2 \pi k/q) = 2 (1 + \cos(4 \pi k/q)) = 4 - 16(\pi k/q)^2 + {\cal O}(q^{-4})$.   Thus, 
\beqas
(x_1 - x_2) & =  & \left( \frac{1}{4} - \frac{1}{2(b+1)}\right) \left(\frac{48 \pi^2}{q^2} \right) +  \left(1 - \frac{b}{2}\right) \left(\frac{12 \pi^2}{q^2} \right) + {\cal O}\left( b q^{-4} \right) \\
& = &  \frac{-6 \pi^2 b}{q^2} + \frac{24 \pi^2}{q^2} + {\cal O}\left(b q^{-4} \right).
\eeqas

We return to $\beta_1 - \beta_2$ and expand using the cosine Taylor expansions.

$$
\beta_1 - \beta_2 = \frac{\alpha_1 - \alpha_2}{2}\sqrt{\frac{b-1}{b+1}} = \frac{6\pi^2}{q^2}\sqrt{\frac{b-1}{b+1}} + {\cal O}(q^{-4})
$$

Lastly, $(y + a)^{3/2}$ is ${\cal O}(b^3)$ and $(x_1+ a)^{1/2}$ is ${\cal O}(b)$, so

$$
\lambda^{(1)}_1 -  \lambda^{(2)}_1 = \frac{6 \pi^2 }{q^2 (b + 1)} + \frac{-3 \pi^2 b}{q^2 \sqrt{x_1 + a} \sqrt{b^2 - 1}} + 
\frac{12 \pi^2}{q^2 \sqrt{x_1 + a} \sqrt{b^2 - 1}}
+{\cal O}(b^{-1} q^{-4}).
$$
As $b \rightarrow \infty$ we see $2 b^{-1} \sqrt{x_1 + a} \rightarrow 1$, so the first two terms cancel asymptotically
and the third term is $\bigo{q^{-2}b^{-2}}$.
\end{proof}

\begin{thm} {\bf (ROC Asymptotic Fiedler Eigenvalue)}
\label{thm:fiedler} Let $b$ and $q$ be large and the same order.   For graph $\Roc$ the smallest nonzero eigenvalue of $\Lh$ is ${\cal O}(b^{-2}q^{-2})$.
\end{thm}
\begin{proof}
The eigenvalue of interest is $\lambda_2(\Lh) = 1- \lambda_2(\Ah) =  1-\lambda^{(1)}_1$.    Define scalar function $g(z) = \sqrt{z+1/4}$.   Using the formulas in Theorem~\ref{thm:Epairs} and a bit of algebra, we see
\beq
\label{eqn:eval_roc}
1 - \lambda^{(1)}_1= 
\frac{1}{b-1} - \frac{\xi^{(1)}_1}{\sqrt{b^2 -1}} =  
 \frac{1}{b-1}  - \frac{\beta_1 + \sqrt{\beta_1^2  + (b-1)}}{\sqrt{b^2 -1}} = \frac{1}{b-1} - \frac{\beta_1}{\sqrt{b^2 -1}} - g(z)
\eeq
with 
$$
z = \frac{(1-\alpha_1)b}{2(b^2 -1)} + \frac{2\alpha_1-3}{2(b^2 -1)} +  \frac{1}{4(b-1)^2} + \frac{\alpha_1^2}{4(b+1)^2}.
$$
We derive the result by demonstrating that the larger terms in (\ref{eqn:eval_roc}) cancel.   Expanding the second term in in the right-hand-side yields
$$
\frac{\beta_1}{\sqrt{b^2 -1}} = \frac{\alpha_1}{2(b+1)} - \frac{1}{2} + \frac{1}{2(b-1)} = -\frac{1}{2}  + \frac{(1+\alpha_1)b}{2(b^2 - 1)} + \frac{1-\alpha_1}{2(b^2 - 1)} .
$$

A second-order Taylor expansion of $g(z)$ at zero shows for each $z$, we see
$g(z) = g(0) + g'(0) z + \frac{1}{2}g''(0) z^2 + \frac{1}{6}g'''(0) z^3+ \frac{1}{24}g^{(iv)}(y) z^4= \frac{1}{2} + z - z^2 +2 z^3 + {\cal O}(z^4) $, where  
$$
g'(z) = \frac{1}{2\sqrt{z + 1/4}} ,
\qquad
 g''(z) = -\frac{1}{4 (z + 1/4)^{3/2}},
\qquad \mbox{and} \qquad
 g'''(z) = \frac{3}{8 (z + 1/4)^{5/2}}.
$$ 
Plugging in, we see the terms of $g(z)$ up to order $b^{-2}$ are
$$
\frac{1}{2} + \frac{(1-\alpha_1)b}{2(b^2 -1)} 	
 + \frac{2\alpha_1-3}{2(b^2 -1)} +  \frac{1}{4(b-1)^2} + \frac{\alpha_1^2}{4(b+1)^2}
- \frac{(1-\alpha_1)^2 b^2}{4(b^2 -1)^2}.
$$

The constant terms in $\beta_1/\sqrt{b^2 -1}$ and $g(z)$ cancel.   The order $b^{-1}$ terms in $1 - \lambda^{(1)}_1$ cancel to an order $b^{-2}$ term,
$$
\frac{1}{b-1}  -\frac{(1+\alpha_1)b}{2(b^2 - 1)} - \frac{(1-\alpha_1 )b}{2(b^2 - 1)} 
= \frac{1}{b-1}  -\frac{b}{b^2 - 1}  = \frac{1}{b^2 - 1} .
$$
Combining fractions the order $b^{-2}$ terms in $1 - \lambda^{(1)}_1$, are reduced to 
\beqas
&&  \frac{1}{b^2 - 1} - \left[ 
 \frac{1-\alpha_1}{2(b^2 - 1)} +
 \frac{2\alpha_1-3}{2(b^2 -1)} +  \frac{1}{4(b-1)^2} + \frac{\alpha_1^2}{4(b+1)^2}
- \frac{(1-\alpha_1)^2 b^2}{4(b^2 -1)^2}
 \right] \\
 & = & \frac{(2 -  \alpha_1) b^2}{(b^2 - 1)^2} + \frac{(2 \alpha_1^2 -3) b}{2(b^2 - 1)^2} + \frac{-\alpha_1^2 + 2\alpha_1 - 9}{4(b^2-1)^2} \\
\eeqas
Factoring in the other $b^{-3}$ terms and the cosine expansion $\alpha_1 = 2-4(\pi /q)^2+{\cal O}(q^{-4})$, we see
$$
\frac{4 \pi^2 b^2}{q^2 (b^2 - 1)^2} + {\cal O}\left(b^{-4} +  b^{-2} q^{-4} \right).
$$
\end{proof}

\newcommand{\condx}{\conductance{\Dneghalf\bx}}
\begin{thm}
    For any vector $\bx\in\blendspacemat$ of $\Roc$, $ \condx \le q\conductance{\fiedvec}$.
\end{thm}
\begin{proof}
    For any $\bx \in \cX_2$, let $S, \bar{S}$ be the partition given by the optimal sweep cut of $\Dneghalf\bx$.
    Fiedler's nodal domain theorem implies at least one of $S, \bar{S}$ is a connected component.
    Because the eigenvectors are block constant, all vertices of each clique are assigned to the same side of the
    partition.
    These imply that $E\paren{S,\bar{S}}=2$. 
    The optimal conductance partition is found when $\vol{S}=\vol{\bar{S}}= q \bchoosetwo$.
    Thus $\condG = \conductance{\bv} = 2\paren{q \bchoosetwo}^{-1}$.

    For any $\bx\in\blendspacemat$, the optimal sweep cut of $\Dneghalf\bx$ will partition the graph into two pieces one containing $k$ blocks
    and the other containing $n-k$ blocks for some $k\le \frac{n}{2}$. 
    That is $\min\paren{{\vol{S}},{\vol{\bar{S}}}}=k\bchoosetwo$.
    Since only edges between cliques are cut, $E(S,\bar{S}) \le 2k$.
    Thus $\condx \le 2\bchoosetwo^{-1}$. By assigning adjacent blocks to alternating sides of the partition,
    we see the bound is tight.

\end{proof}

Notice that the smallest eigenvalue of $\Lh$ scales as $\bigo{b^{-2}q^{-2}}$,
  but the optimal conductance scales as $\bigo{b^{-2}q^{-1}}$,
  and that the worst case sweep cut partition of a blend has conductance $2\bchoosetwo^{-1}$ independent of $q$.
The remainder of this section shows that by accepting this factor of $q$ in conductance, one gains tremendously in
computational efficiency.

\newcommand{\orthproj}{\paren{{\bf I}-\spectralproj}}
\newcommand{\spectralproj}{\Proj{\blockconstspace}}
\newcommand{\minperbub}{b^{-\half}}
\subsection{A residual tolerance for \roc{}}\label{sec:rocbygeneral}
In order to derive a residual tolerance for the \roc{}, we show that for any vector in $\blendspace$
at least one block is sufficiently far from the other blocks in spectral coordinates. 
\begin{lemma}\label{thm:mus-apart}
    Define $\block{i} = \set{ib+2\dots ib+q}$ as the vertex blocks of $\Roc$.
    For any vector $\bx$, let $\mean{j} = \abs{\block{i}}^{-1}\sum_{i\in\block{j}}x_i$.
    Let $\blockconstspace$ be the span of 
    $\set{\be_{ib+1}\mid i \in 0\dots q-1} \cup \set{\blockindicator{i}{q}\mid  i \in 0\dots q-1}$.  
    For any vector $\bx\in \blockconstspace$, $\norm{\means}_\infty > n^{-1/2}\norm{\bx}_2$.
\end{lemma}
\begin{proof}
    By construction of $\bx$, $\mean{i} = x_j$ for all $j\in\block{i}$.
    Thus $\norm{\bx}_\infty = \norm{\means}_\infty$.
    Equivalence of norms implies $\norm{\bx}_2 < \sqrt{n}\mean{q}$.
\end{proof}

We are able to apply \cref{thm:blendspaceconductance} and derive a residual tolerance for recovering the \roc{}. 
\begin{coro}\label{thm:roc-residbound-general}
    If $\bx$ is an approximate eigenvector of $\Roc$ with eigenresidual less than $\frac{C}{q\sqrt{n}}$ for some constant $C$ and $\bx^t\ones=0$,
    then $\conductance{\Dneghalf\bx}\le 2 \bchoosetwo^{-1}$
\end{coro}
\begin{proof}
In the setting of \cref{thm:blendspaceconductance}, 
  choose $G=\Roc$, $\psi=2 \bchoosetwo^{-1}$.
In the notation of \cref{thm:mus-apart} applied to sorted $\bv$,
for all  $\bv \in V$, $g_v = \max_i \paren{\mean{i+1}-\mean{i}} \ge q^{-1}\norm{\means}_\infty \ge \paren{q\sqrt{n}}^{-1}$. 
For some $C\in\bigo{1}$, $\blendgap g > \frac{C}{q\sqrt{n}}$.
So \Cref{thm:blendspaceconductance} implies computing $\bx$ to a residual tolerance of $\frac{C}{q\sqrt{n}}$ is sufficient to guarantee $\conductance{\bx} \le 2\bchoosetwo^{-1}$.
\end{proof}
\Cref{thm:roc-residbound-general} gives a sufficient condition on approximate eigenvectors of $\Roc$ such that $\bx$ partitions the graph at least as well as any partition that recovers the cliques.
\Cref{thm:minimal-perturbation-lb} and \Cref{thm:minimal-perturbation-ub} using analysis specialized for $\Roc$ in
\cref{sec:roc-perturbation} to construct the minimal perturbation that causes the sweep cut proceedure to fail.



\subsection{Minimal Perturbation}\label{sec:roc-perturbation}
We want to find the minimal error at which a vector can make a mistake.
The effects of the corner vertices only enter into the constants of the following results,
  and for clarity of exposition we omit handling of the corner vertices.
\Cref{thm:minimal-perturbation-lb} shows that no perturbation with norm less than $\minperbfrac$ can induce a mistake in the sweep cut.
\Cref{thm:minimal-perturbation-ub} constructs a perturbation that induces a mistake in the sweep cut and has norm less than $\minperbub$.
For the parameter regime $q\in\bigo{1}$, the bounds in \Cref{thm:roc-residbound-general}, \Cref{thm:minimal-perturbation-lb}, and \Cref{thm:minimal-perturbation-ub} are all equivalent up to constant factors.
Using the same notation as \cref{thm:mus-apart},
we say that a vector $\by$ recovers all the cliques of $\Roc{}$
    if there is a threshold $t\in\opintervalop{\minelt{y}}{\maxelt{y}}$ 
    such that for all $\block{i}$ with $j,k\in\block{i}$,
    $y_{j} < t$ if and only if  $y_{k} < t$.

\begin{thm}\label{thm:minimal-perturbation-lb}
    Let $\blockconstspace$ be defined as in \cref{thm:mus-apart},
    and let $\Proj{\blockconstspace}$ be the orthogonal projector onto $\blockconstspace$. 
For any vector $\by$ orthogonal to $\ones$, define $\errs=$ $\orthproj\by$.
If $\norm{\errs}_2\le\minperbfrac\norm{\by}_2$,
    then $\by$ recovers all the cliques of $\Roc{}$.
\end{thm}
\begin{proof}
    Define $\bx = \Proj{\blockconstspace}\by$.
    Without loss of generality, relabel the vertices and blocks such that $\mean{i} \le \mean{i+1}$.
    Let $\alpha_i = \maxelt{\err{}}$ and $\beta_i = \minelt{\err{}}$ for each $\block{i}$.
    Note that $\alpha_i > 0$ and $\beta_i < 0$ since $\errs \perp \blockindicator{i}{q}$.
    The vector $\by$ recovers all the cliques if and only if there is a $\block{i}$ where
    $\alpha_{i} - \beta_{i+1} \le \mean{i+1} - \mean{i} $.
    In this case, a threshold can be chosen in
        $\opintervalop{
            \mean{i}+\alpha_{j}}{
            \mean{i+1}+\beta_{k}}
        $.
    Suppose that $\by$ does not recover all the cliques,
    then for all $\block{i}$ $\alpha_{i} - \beta_{i+1} > \mean{i+1}-\mean{i}$.
    This implies $\sum_{i=0}^{q-1} \alpha_{i}-\beta_{i+1} > \sum_{i=0}^{q-1} \mean{i+1}-\mean{i}$.
    Thus  we can bound the 1-norm error as follows:
    \[
        \norm{\errs}_1
        \ge \sum_i\paren{\alpha_{i} - \beta_{i+1}}
        \ge \mean{q} - \mean{1}
        \ge n^{-\half}\norm{\bx}_2.
    \]
    Since $\errs$ must have at least $q$ nonzero entries $\norm{\errs}_2 > \paren{2qn}^{-\half}\norm{\bx}_2$.
    Applying $\norm{\by}^2 = \norm{\bx}^2 + \norm{\errs}^2$, we see that
    ${\paren{1+2qn}^{-\half}}\norm{\by} < \norm{\errs}_2$.

\end{proof}

The proof of Theorem~\ref{thm:minimal-perturbation-lb} yields a construction
for the minimal perturbation of $\bx$ that does not recover all the cliques. 

\begin{thm}\label{thm:minimal-perturbation-ub}
    For any unit vector $\bx\in\blockconstspace$ orthogonal to $\ones$,
    there exists a perturbation $\errs$ where $\norm{\errs} < b^{-\half}, \Proj{\blockconstspace}\errs = 0$
        such that $\by = \bx+\errs$ does not recover all the cliques.
\end{thm}
\begin{proof}
For any $\bx\in\blockconstspace$, 
set $\alpha_0=0, \beta_{q-1}=0, \alpha_i=-\beta_{i+1} = \frac{\mean{i+1}-\mean{i}}{2}$.
\[
    \norm{\errs}_2^2
    = \sum_{i=0}^{q-1} \alpha_i^2 + \beta_{i+1}^2
    = \half\sum_{i=0}^{q-1}\paren{\mean{i+1}-\mean{i}}^2
    = b^{-1}\norm{\bx}_2^2 -2\sum_{i=0}^{q-1}\mean{i+1}\mean{i}
    < b^{-1}\norm{\bx}_2^2
\]
\end{proof}
Theorem~\ref{thm:minimal-perturbation-lb} implies that $\minperbfrac$ is a sufficient accuracy to ensure recovery of all the cliques, and Theorem~\ref{thm:minimal-perturbation-ub} implies that for some elements of the top invariant subspace
accuracy less than $\minperbupperbound$ is necessary to ensure recovery of all the cliques from that vector.

\Cref{fig:minimal-perturbation-lowerbounds} lends validation to the formulas in \cref{thm:minimal-perturbation-lb}.
The experiment shown is to take a random (gaussian unit norm) linear combination of $\cX_{signal}$, and then construct the minimal perturbation that makes a mistake.
\Cref{fig:minimal-perturbation-lowerbounds} shows the minimum over all samples as a function of $n$.
This experiment is conducted for three different parameter regimes, $q=25$, $b=25$, and $b=q=\sqrt{n}$.
One can see that the lower bound from \cref{thm:minimal-perturbation-lb} is below the empirical observation,
and that this lower bound is within a constant factor of the observed size of the minimal perturbation.
\begin{figure}
    \centering
    \includegraphics[width=0.9\textwidth]{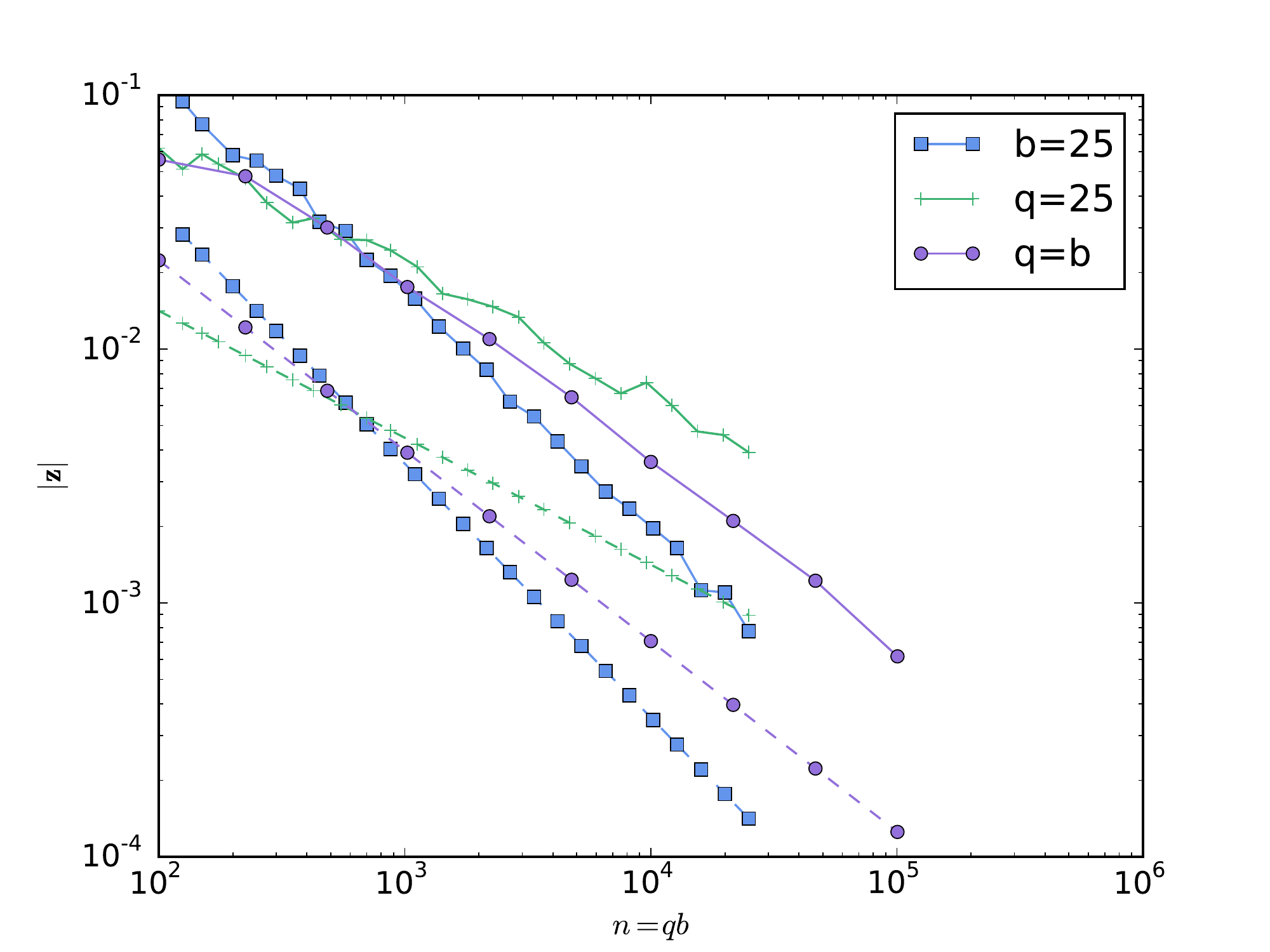}
    \caption{Empirical measurements of minimal error perturbations on a log-log scale. Lower bounds are shown in the same color with dashed lines.}\label{fig:minimal-perturbation-lowerbounds}
\end{figure}

We can now apply Theorem~\ref{thm:singleevecerror} and Theorem~\ref{thm:blenderror}
to determine the residual tolerance for an eigensolver for graph paritioning. 
The residual tolerance can be no larger than that for the the ring of cliques,
but some graphs may require smaller tolerances.

\subsection{The Power Method}\label{sec:powermethod}


\newcommand{\erf}{\mbox{erf}}
\newcommand{\kstar}{k^*}
\newcommand{\theprob}{1-\zeta}
\newcommand{\lbar}{\lambda^*}
From the eigenvalues and error tolerances above, one can determine an upper bound on the number of iterations required by the power method to recover all the cliques in $\Roc{}$.

\begin{thm}\label{thm:rocpowermethod}
    Let $\bx_0$ be sampled from $\cN_n\paren{0, 1}$.
    Let $\bx_k$ be the $k$-th iteration of the power method, $\bx_{k+1} \longleftarrow \Ah \bx_k / \|\bx_k\|$ for $\Roc{}$. 
    Let $\zeta = \left(\frac{e^{7/8}}{8}\right)^{q/2} + \left(\frac{2}{e}\right)^{(n-q)/2}$. 
    There is a $\kstar$ of $\bigo{\log_b q}$ such that 
    for $k \geq \kstar$ a sweep cut based on $\bx_k$ makes no errors
    with probability at least $\theprob$. 
\end{thm}
\begin{proof}
First, we bound $\|\orthproj \bx_0\|^2$ and $\|\spectralproj \bx_0\|^2$ probabilistically.   Each entry in $\bx_0$ is independently sampled from $\cN\paren{0, 1}$.   For any orthonormal basis of $\mathbb{R}^n$, $\{\bv_k\}_{k=1}^n$, the distribution of each $\bv_k^t \bx_0$ is also $\cN\paren{0, 1}$.   Therefore, the distribution of $\|\spectralproj \bx_0\|^2 $ is a $\chi^2$-distribution of order $q$, which has expected value $q$ and cumulative distribution function $\gamma(q/2, z/2)/\Gamma(q/2)$, where $\Gamma(\cdot)$ is the gamma function and $\gamma(\cdot, \cdot)$ is the lower incomplete gamma function.    Let $c_0 \in (0,1)$, using Chernoff bounds we have
$$
p_0 := \mbox{Prob} \left[ \|\spectralproj \bx_0\|^2 > c_0 q \right] = 1- \frac{\gamma\left(\frac{q}{2}, \frac{c_0 q }{2}\right)}{\Gamma\left(\frac{q}{2}\right)} \geq 1-\left(  c_0 e^{1 - c_0} \right) ^{q/2}
$$

Similarly, $\|\orthproj \bx_0\|^2 $ is from a $\chi^2$-distribution  of order order $n-q$, with expected value $n-q$ and known cdf.    Let $c_1 \in (1,\infty)$, we have
$$
p_1 := \mbox{Prob} \left[ \|\spectralproj \bx_0\|^2 < c_1 (n-q) \right] = \frac{\gamma\left(\frac{n-q}{2}, \frac{c_1 (n-q) }{2}\right)}{\Gamma\left(\frac{n-q}{2}\right)} \geq 1-\left(  c_1 e^{1 - c_1} \right) ^{(n-q)/2}
$$

The union of events $\left[ \|\spectralproj \bx_0\|^2 > c_0 q \right]$ and $\left[ \|\spectralproj \bx_0\|^2 < c_1 (n-q) \right]$ is a subset of all possibilities for which $\left[ \frac{ \|\orthproj \bx_0\|^2 }{\|\spectralproj \bx_0\|^2} < \frac{c_1 (n-q)}{ c_0 q} \ \right]$ holds.   Therefore, setting $c_0 = 1/8$ and $c_1 = 2$, we see

$$
\mbox{Prob} \left[ \frac{ \|\orthproj \bx_0\|^2 }{\|\spectralproj \bx_0\|^2} < \frac{c_1 (n-q)}{ c_0 q} \ \right] > p_0 p_1 >1 - \zeta
$$
where $\zeta := \left(\frac{e^{7/8}}{8}\right)^{q/2} + \left(\frac{2}{e}\right)^{(n-q)/2}$ is a small positive constant
when $q,b > 4$.
Because a sweep cut does not depend on the norm of a vector, we consider the iteration, $\bx_{k} \longleftarrow
\frac{1}{\lambda_q}\Ah \bx_{k-1}$ which is equivalent to the power method.   
Letting $\lbar = \pairmax{\abs{\lambda_{q+1}}}{\abs{\lambda_n}}$,
this iteration accentuates vector
components in the range of $\spectralproj$ by a factor greater than 1 and attenuates those orthogonal to this space by
factors less than $\lbar/ \lambda_q$.   If $\|\spectralproj \bx_0\|^2 > c_0 q$, then
   
$$\|\bx_k\|^2 \geq \|\spectralproj \bx_k\|^2 \geq \|\spectralproj \bx_0\|^2 \geq c_0 q.$$

Also, if $\|\orthproj \bx_0\|^2 \leq c_1 (n-q)$, then

$$  \|\orthproj \bx_k\|^2 \leq \left( \frac{\lbar}{\lambda_q}\right)^{2k}  \|\orthproj \bx_0\|^2 \leq \left(
\frac{\lbar}{\lambda_q}\right)^{2k} c_1 (n-q).$$    

Therefore, under the assumptions on $\bx_0$, the $k$-th iteration satisfies

$$
\frac{\norm{\orthproj \bx_k}}{\norm{\bx_k}} \leq \left( \frac{\lbar}{\lambda_q}\right)^{k} \sqrt{\frac{c_1}{c_0} (b-1)}
= 4 \left( \frac{\lbar}{\lambda_q}\right)^{k} \sqrt{b-1}.
$$
    
By \cref{thm:minimal-perturbation-lb}, if this ratio is less than $\paren{1+2qn}^{-1/2}$, then $\bx_k$ makes no errors. We see this is ensured by
    $$
    k \geq \kstar := \left\lceil \frac{ \log 4 + \log (b-1) + \log(1+2qn)}{2\log\left( \lambda_q / \lbar\right)} \right\rceil.
    $$
Revisiting Corollary~\ref{cor:evals} we that $\lambda_q > 1 - C_q/b^2$ and $\lbar < \max(C_{q+1}, C_n)/b$ so
$\lambda_{q} / \lbar = C^* b$, where $C^*$ is an order 1 constant.   Plugging this in we see that $\kstar$ is in $\bigo{\log_b q}$.
    

\end{proof}

\subsection{Experiment}\label{sec:roc_experiment}
Here we show the results of a numerical experiment in order to lend intuition and validation to the theorems.
Take $\Roc{}$ and a random seed vector $\xnaught$.
Then apply power iteration $\bx^{(i)} = (\ahatshifted) \bx^{(i-1)}$.
Far $b=20$ and $q=30$ the relevant measures of convergence are shown in \cref{tab:pmeth}.
\Cref{fig:pmeth} illustrates the convergence behavior in terms of the conductance of all sweep cuts,
	and the reordered adjacency matrix represented in sparsity plots.
\Cref{tab:pmeth} shows that the convergence to the Fiedler vector stalls after iteration 3,
	but convergence to the space orthogonal to $\cX_{noise}$ continues unabated.   Letting $\Pi$ be the projection onto $\cX_{noise}^\perp$, we measure $\|\Pi \bx^{(i)}\|$ for each iteration.
Applying \cref{thm:rocpowermethod}, we calculate that $k^* =5$ iterations will perfectly resolve the clique structure with probability at least $1-\zeta = 0.99999998575$.  After one iteration the sweep cut did not split any cliques, but only a single clique is shaved off.   After 3 iterations a nearly optimal partition is found.   


\begin{minipage}[h]{\textwidth}
\centering
\begin{tabular}{ccc}
First Iteration & Second Iteration & Third Iteration \\
\includegraphics[width=132pt]{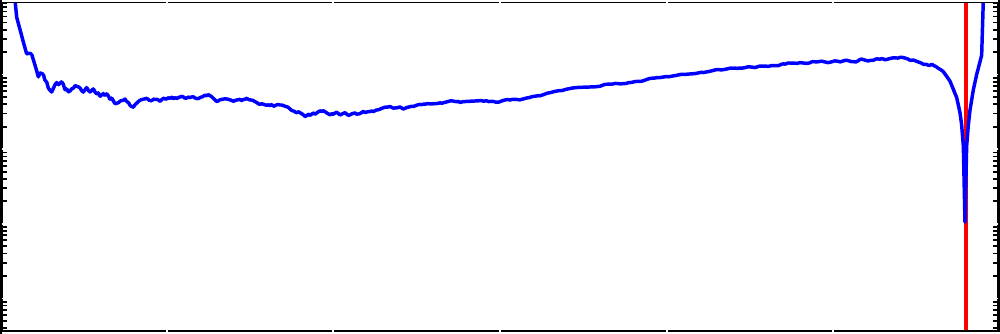}&
\includegraphics[width=132pt]{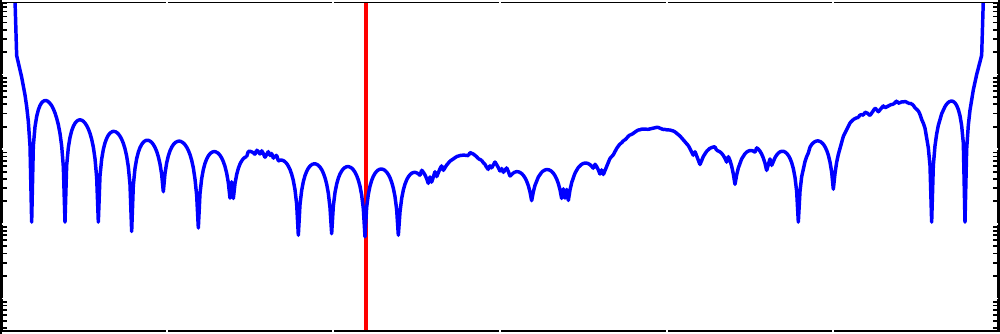}&
\includegraphics[width=132pt]{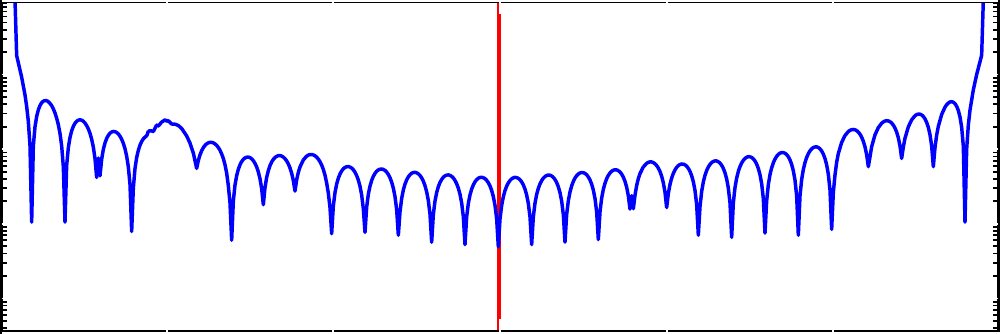}\\
\includegraphics[width=132pt]{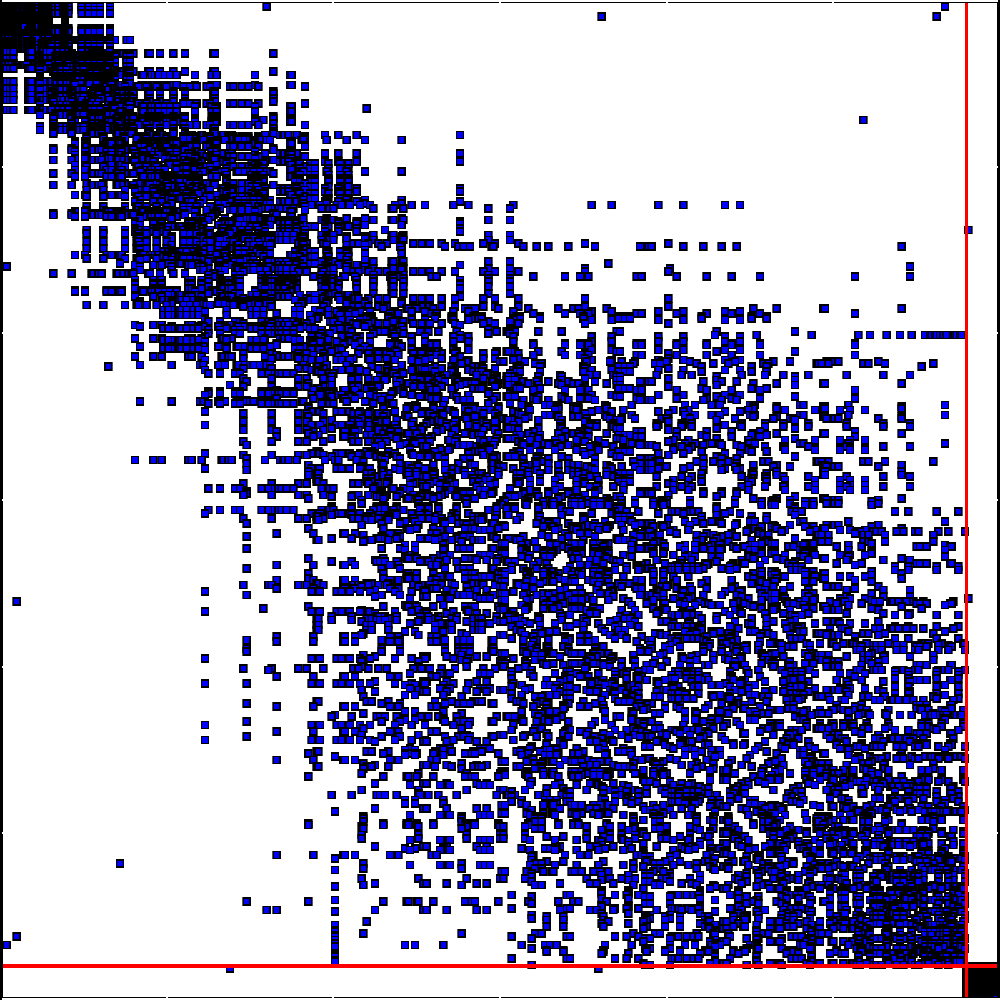}&
\includegraphics[width=132pt]{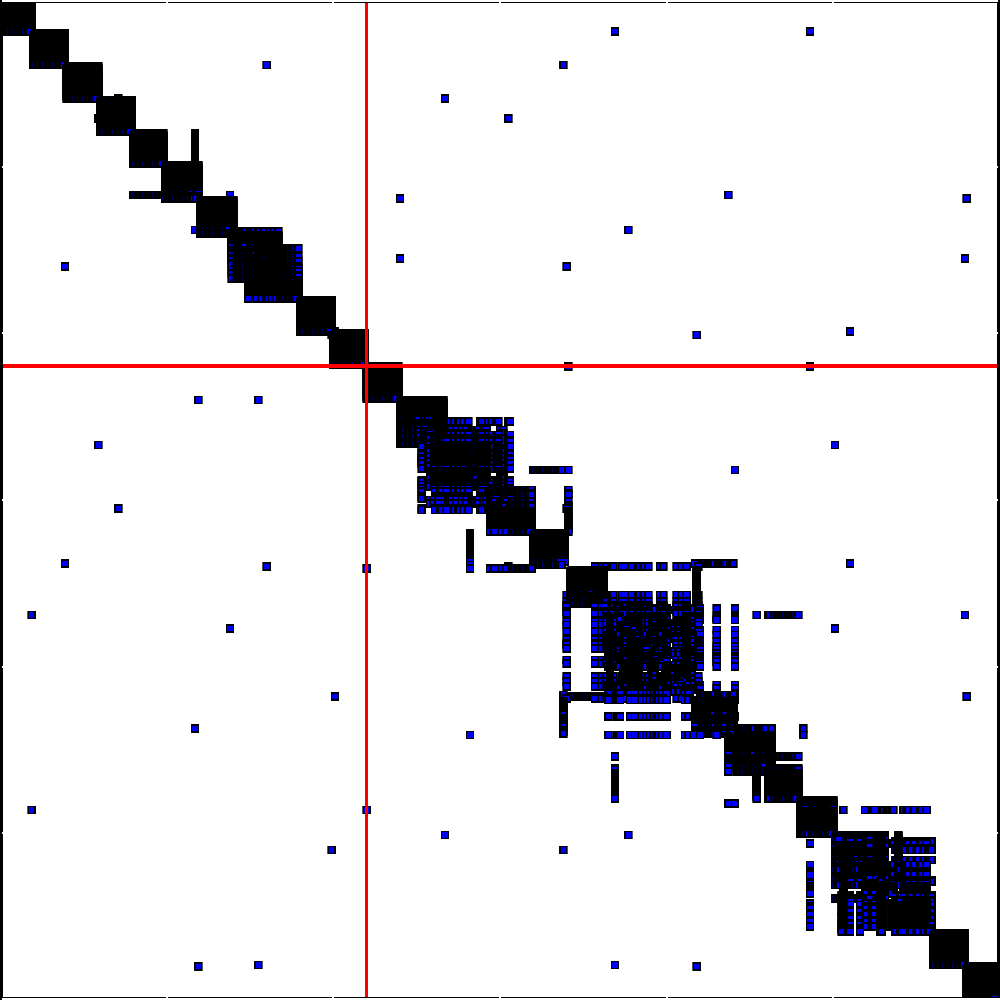}&
\includegraphics[width=132pt]{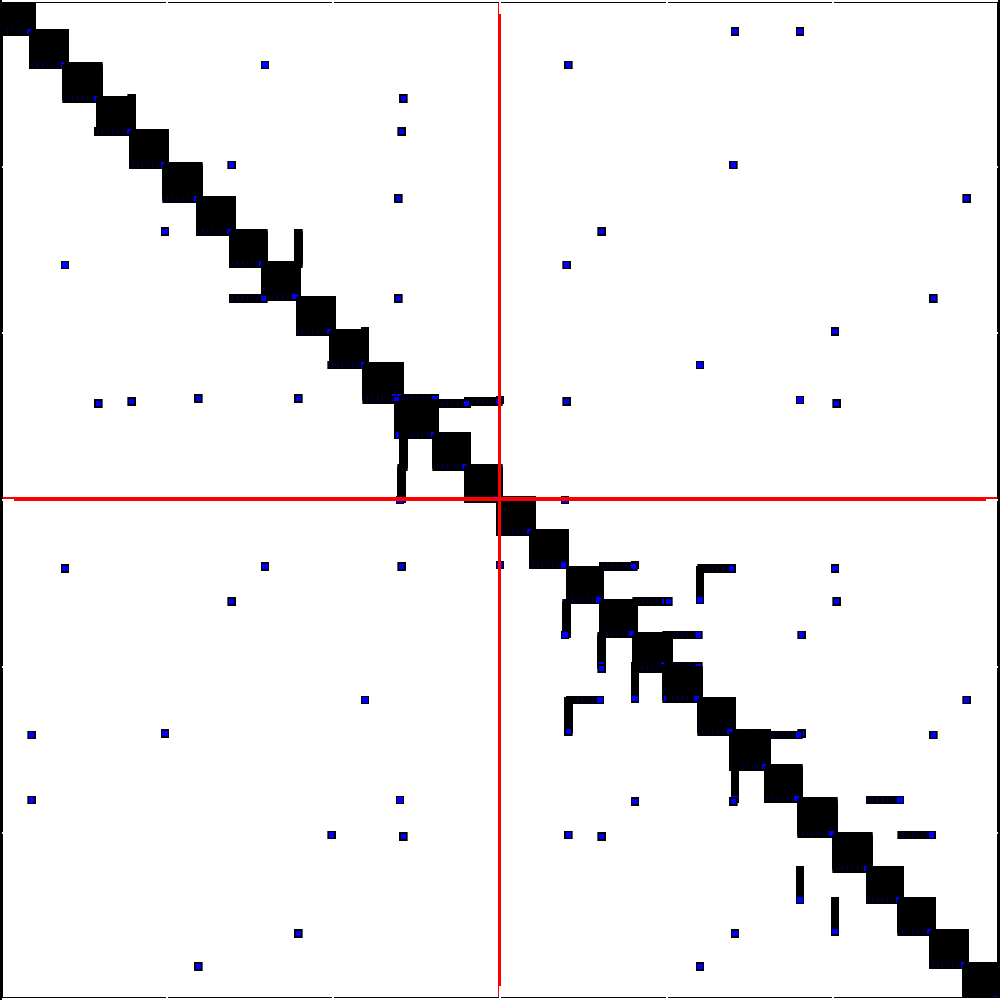}
\end{tabular}
\captionof{figure}{First several iterations of the power method applied to ${\cal R}_{b=20,q=30}$.    Above: sweep conductance of $A$ reordered by sorting the 1$^{st}$ (top-left), 2$^{nd}$ (top middle), and 3$^{rd}$ iterations (top-right).   Horizontal axis represents which vertex to split at under the induced ordering; vertical axis is the conductance for each split on a log scale.   Below: matrix sparsity plots of $A$ reordered by sorting the 1$^{st}$ (bottom-left), 2$^{nd}$ (bottom-middle), and 3$^{rd}$ iterations (bottom-right).  Red lines demonstrate which edges are cut for the optimal cut in each ordering.   }
\label{fig:pmeth}

\vspace{.5cm}
\begin{tabular}{|c|c|c|c|c|c|}
\hline
      $i$ & $\epsilon$  & $\| \Pi \bx^{(i)} \|$ &    $\mu$    & $\phi(\bx)$ & $\sqrt{2 \mu}$ \\ \hline
        0 & 1.24806e+01 & 2.25433e+01 & 1.52966e+00 & 2.03094e+00 & 1.74909e+00   \\ 
        1 & 9.20534e-01 & 5.02679e-02 & 8.37276e-01 & 1.05263e-02 & 1.29404e+00   \\ 
        2 & 1.02629e-01 & 1.45977e-02 & 5.67751e-02 & 6.68577e-03 & 3.36972e-01   \\ 
        3 & 1.08242e-02 & 8.10095e-04 & 1.01284e-02 & 4.89853e-03 & 1.42326e-01   \\ 
        4 & 1.01568e-02 & 4.28591e-05 & 9.87761e-03 & 4.89853e-03 & 1.40553e-01   \\ 
        5 & 1.01262e-02 & 2.26698e-06 & 9.85062e-03 & 4.89853e-03 & 1.40361e-01   \\ 
        6 & 1.01022e-02 & 1.19907e-07 & 9.82549e-03 & 4.89853e-03 & 1.40182e-01   \\ 
        7 & 1.00782e-02 & 6.34217e-09 & 9.80038e-03 & 4.89853e-03 & 1.40003e-01   \\ 
        8 & 1.00543e-02 & 3.35448e-10 & 9.77526e-03 & 4.89853e-03 & 1.39823e-01   \\ 
        9 & 1.00303e-02 & 1.77422e-11 & 9.75013e-03 & 4.89853e-03 & 1.39643e-01   \\ 
       10 & 1.00063e-02 & 9.38388e-13 & 9.72498e-03 & 4.89853e-03 & 1.39463e-01   \\ 
 $\cdots$ & $\cdots$ & $\cdots$ & $\cdots$  & $\cdots$ & $\cdots$ \\ 
 $\infty$ & 0.00000e+00 & 0.00000e+00 & 1.14176e-04 & 3.50018e-04 & 1.51113e-02   \\
&  && $\lambda_2(\Lh)$ & $\phi(\bv_2) = \phi_{\cG}$ & $\sqrt{2 \lambda_2(\Lh)}$   \\  \hline
\end{tabular}
\captionof{table}{Table corresponding to Figure~\ref{fig:pmeth} with 10 iterations.  Convergence to the Fiedler eigenpair is slow, yet convergence to the orthogonal complement of $\cX_{noise}$ is rapid (column 2). }
\vspace{.5cm}

\label{tab:pmeth}
\end{minipage}


\section{Conclusions}

\renewcommand{\LSEVs}{locally supported eigenvectors}
When partitioning graphs where the spectral gap is small, computation of an accurate approximation to a Fiedler vector is difficult.
In order to satisfy the needs of spectral partitioning without computing eigenvectors to high accuracy,
we introduce spectral blends and in particular the \emph{blend gap}.
Section~\ref{sec:convergeblends} controls the distance between an approximate eigenvector and an invariant subspace in
terms of the eigenresidual and blend gap thereby showing that accurate approximation to a spectral blend is easier to
compute than an accurate approximation of a single eigenvector.
We provide a general tool for deriving residual tolerances based on the structure of the graph spectrum.
In order to illustrate the utility of spectral blends, 
  Section~\ref{sec:ROCstart} studies a model problem and uses the theory of block cyclic matrices and \LSEVs{} to present a closed form for the eigenvalues and vectors.
We show that any blend of large eigenvalue eigenvectors for the \roc{} recovers a correct clustering.
This indicates that for problems where there are multiple good partitions of the graph, spectral blends can be used to partition accurately.
The eigendecomposition of the model problem provides error and residual tolerances for solving this problem with sweep
cuts.
\Cref{thm:minimal-perturbation-lb} allows us to give guidance for error tolerances for spectral partitioning.
One should solve the eigenproblem to a tolerance no greater than $\bigo{n^{-1}}$ for graphs of size $n$.
\Cref{thm:rocpowermethod} shows that for the ring of cliques where the number of clusters is polynomial in the sizes of the clusters, the number of power method steps taken to recover the
clusters is \bigo{1}.
Further research will be able to expand these results to more general graphs which have multiple good partitions.

\bibliography{references,citations}{}
\bibliographystyle{plain}
\end{document}